\documentclass{article}

\usepackage{amsfonts}
\usepackage{amssymb}
\usepackage{amsmath}
\usepackage{amsthm}
\usepackage{ifthen}
\usepackage{enumerate}
\usepackage{tikz}

\newcommand{\C}{{\mathbb C}}
\newcommand{\N}{{\mathbb N}}
\newcommand{\PP}{{\mathbb P}}

\newcommand{\jac}{{\mathcal J}}
\newcommand{\hess}{{\mathcal H}}
\newcommand{\grad}{\nabla}
\newcommand{\bi}{\mathfrak{b}}
\newcommand{\rk}{\operatorname{rk}}
\newcommand{\tr}{\operatorname{tr}}
\newcommand{\GL}{\operatorname{GL}}
\newcommand{\trdeg}{\operatorname{trdeg}}
\newcommand{\tp}{^{\rm t}}

\newcommand{\imp}{{\mathversion{bold}$\Rightarrow$ }}
\newcommand{\bcdot}{$\discretionary{\mbox{$ \cdot $}}{}{}$}
\newcommand{\parder}[3][Default]{
	\frac{\partial \ifthenelse{\equal{#1}{Default}}{}{^{#1}}#2}{
              \partial #3 \ifthenelse{\equal{#1}{Default}}{}{^{#1}}}}

\theoremstyle{plain}
\newtheorem{theorem}{Theorem}[section]
\newtheorem{proposition}[theorem]{Proposition}
\newtheorem{lemma}[theorem]{Lemma}
\newtheorem{corollary}[theorem]{Corollary}
\theoremstyle{definition}
\newtheorem{definition}[theorem]{Definition}
\theoremstyle{remark}
\newtheorem{remark}[theorem]{Remark}
\newtheorem{example}[theorem]{Example}
\theoremstyle{plain}

\allowdisplaybreaks
\numberwithin{equation}{section}

\makeatletter
\newcommand{\nolisttopbreak}{\vspace{\topsep}\nobreak\@afterheading}
\makeatother

\newenvironment{listproof}[1][\proofname]{\begin{proof}[#1]\mbox{}\nolisttopbreak}{\end{proof}}

\title{Homogeneous quasi-translations in dimension $5$\footnote{Many
of the results of this article appeared already in Chapter 3 in the author's Ph.D. thesis
\cite{homokema}.}}
\author{Michiel de Bondt\footnote{The author's Ph.D. project was supported by The Netherlands
Organization for Scientific Research (NWO).}
\footnote{Department of Mathematics, Radboud University,
Postbus 9010, 6500 GL Nijmegen, The Netherlands}
\footnote{E-mail: M.deBondt@math.ru.nl}}

\begin{document}

\maketitle

\begin{abstract}
We give a proof in modern language of the following result by Paul Gordan and Max N{\"o}ther:
a homogeneous quasi-translation in dimension $5$ without linear invariants would be linearly
conjugate to another such quasi-translation $x + H$, for which $H_5$ is algebraically independent over 
$\C$ of $H_1, H_2, H_3, H_4$. 
Just like Gordan and N{\"o}ther, we apply this result to classify all homogeneous polynomials 
$h$ in $5$ indeterminates, for which the Hessian determinant is zero.

Others claim to have reproved `the result of Gordan and N{\"o}ther in $\PP^4$' as well,
but their proofs have gaps, which can be fixed by using the above result 
about homogeneous quasi-translations. Furthermore, some of the proofs assume that $h$ is 
irreducible, which Gordan and N{\"o}ther did not. 

We derive some other properties which $H$ would have. One of them is that $\deg H \ge 15$,
for which we give a proof which is less computational than another proof of it by Dayan Liu. 
Furthermore, we show that the Zariski closure of the image of $H$ would be an irreducible component of 
$V(H)$, and prove that every other irreducible component of $V(H)$ would be a $3$-dimensional linear 
subspace of $\C^5$ which contains the fifth standard basis unit vector.
\end{abstract}

\paragraph{Key words:} Quasi-translation, Hessian, determinant zero, homogeneous,
locally nilpotent derivation, algebraic dependence, linear dependence.

\paragraph{MSC 2010:} 14R05, 14R10, 14R20.

\section{Introduction}

Throughout this paper, we will write $x$ for an $n$-tuple 
$(x_1,x_2,\ldots,x_n)$ of variables, where $n$ is a 
positive integer. We write $\jac F$ for the Jacobian matrix of a polynomial map 
$F = (F_1,F_2,\ldots,F_m)$ with respect to $x$, where $m$ is another positive 
integer, i.e.\@
$$
\jac F = \left( \begin{array}{cccc}
\parder{}{x_1} F_1 & \parder{}{x_2} F_1 & \cdots & \parder{}{x_n} F_1 \\
\parder{}{x_1} F_2 & \parder{}{x_2} F_2 & \cdots & \parder{}{x_n} F_2 \\
\vdots             & \vdots             &        & \vdots             \\
\parder{}{x_1} F_m & \parder{}{x_2} F_m & \cdots & \parder{}{x_n} F_m
\end{array}\right)
$$
We write $\hess f$ for the Hessian matrix of a polynomial $f$ with respect to $x$,
i.e.\@
$$
\hess f = \left( \begin{array}{cccc}
\parder[2]{}{x_1} f & \parder{}{x_2} \parder{}{x_1} f & \cdots & \parder{}{x_n} \parder{}{x_1} f \\
\parder{}{x_1} \parder{}{x_2} f & \parder[2]{}{x_2} f & \cdots & \parder{}{x_n} \parder{}{x_2} f \\
\vdots             & \vdots             & \ddots & \vdots             \\
\parder{}{x_1} \parder{}{x_n} f & \parder{}{x_2} \parder{}{x_n} f & \cdots & \parder[2]{}{x_n} f
\end{array}\right)
$$
We see a polynomial $f$ as a polynomial with only one component, so
$$
\jac f = \Big( \parder{}{x_1} f ~ \parder{}{x_2} f ~ \cdots ~ \parder{}{x_n} \Big)
$$
and write $\grad f = (\jac f)\tp$. Here, and in the rest of the article,
$(\cdots)\tp$ stands for the transpose matrix. So
$$
\hess f = \jac (\grad f)
$$
Just like with $x$, we will write $y$ for another $n$-tuple $(y_1,y_2,\ldots,y_n)$
of variables. But unlike $x$ and $y$, $t$ will be just a single variable.

\begin{definition}
Let $F = x + H$ be a polynomial map from $\C^n$ to $\C^n$. Then we call $F$ a
{\em quasi-translation} if $2x - F = x - H$ is the inverse of $F = x + H$.
\end{definition}

The condition that $x - H$ is the inverse of $x + H$ is automatically fulfilled
if $\deg H = 0$, in which case $x + H$ is a regular translation. So a quasi-translation
is a polynomial map which is characterized by a property of a regular translation.

Below are some examples of quasi-translations in dimension $n = 4$:
\begin{align*}
x &+ (x_2^2 x_3 - 3x_3^3 x_4 - 5,0,0,0) \\
x &+ (1,x_4,x_4^2,0) \\
x &+ (x_3^2 -3x_3^3 x_4 - 5, x_3 + 7 x_4^7, 0, 0) \\ 
x &+ \big(b(a x_1 - b x_2),a(a x_1 - b x_2), \\*
  &\quad\quad\!\!b(a x_3 - b x_4),a(a x_3 - b x_4)\big) \qquad \mbox { with } a,b \in \C
\end{align*}

In the next section, we will
show that $x + H$ is a quasi-translation, if and only if $\jac H \cdot H = 0$. This is
equivalent to that for the derivation $D = H_1 \parder{}{x_1} + H_2 \parder{}{x_2} + \cdots + 
H_n \parder{}{x_n}$, $D^2 x_i = 0$ for all $i$, because $D^2 x_i = D H_i = \jac H_i \cdot H$.

Hence quasi-translations correspond to a special kind of locally nilpotent derivations.
Furthermore, invariants of the quasi-translation $x + H$ are just kernel elements of
$D$. Paul Gordan and Max N{\"o}ther call these kernel elements `Functionen $\Phi$' in \cite{gornoet}.

In addition, we can write $\exp(D)$ and $\exp (tD)$ for the automorphisms
corresponding to the maps $x + H$ and $x + tH$ respectively. But in order to make the 
article more readable for readers that are not familiar with derivations, we will omit the 
terminology of derivations further in this article.

In \cite{gornoet}, Gordan and N{\"o}ther studied (homogeneous) quasi-translations to obtain results 
about (homogeneous) polynomials $h$ with $\det \hess h = 0$. One such a result is the classification
of homogeneous polynomials in $5$ indeterminates for which the Hessian determinant is zero. 
This classification has been reproved in \cite{franch} and \cite{garrep},
but only for the case where $h$ is an irreducible polynomial. In \cite[Ch.\@ 7]{russo}, the proof of
\cite{garrep} is extended to the case where $h$ is a square-free polynomial. 
With an easy argument, which the reader may find, one can extend these results to the case where $h$ 
is a power of such a polynomial. But then, you still do not have all polynomials $h$.

However, Francesco Russo, the author of \cite{russo}, told me that by way of 
\cite[Th.\@ 2.2]{cilrussim} one can reduce the general case to the case where $h$ is square-free.
This is indeed true, because of the following.

\begin{proposition}
Let $h \in \C[x]$ and let $\tilde{h}$ be the square-free part of $h$.
\begin{enumerate}[\upshape (i)]

\item If $\det \hess h = 0$, then $\det \hess \tilde{h} = 0$.

\item Suppose that $a_1,a_2,\ldots,a_{n-2} \in \C[x_1,x_2]$ are relatively prime.
Let 
$$
A := \C[x_1,x_2,a_1(x_1,x_2)x_3+a_2(x_1,x_2)x_4+\cdots+a_{n-2}(x_1,x_2)x_n]
$$
If $\tilde{h} \in A$, then $h \in A$.

\end{enumerate}
\end{proposition}

\begin{listproof}
\begin{enumerate}[(i)]

\item This is a special case of \cite[Th.\@ 2.2]{cilrussim}.

\item Suppose that $\tilde{h} \in A$, and let $f$ be an arbitrary factor 
of $h$ over $\C[x]$. It suffices to show that $f \in A$. 

Over $\C(x_1,x_2)$, $h$ is a polynomial in the linear form
$a_1(x_1,x_2)x_3+a_2(x_1,x_2)\bcdot x_4+\cdots+a_{n-2}(x_1,x_2)x_n$. Just
like $\C(x_1,x_2)[x_3]$,
$$
\C(x_1,x_2)[a_1(x_1,x_2)x_3+a_2(x_1,x_2)x_4+\cdots+a_{n-2}(x_1,x_2)x_n]
$$
is factorially closed in $\C(x_1,x_2)[x_3,x_4,\ldots,x_n]$.
Consequently, $f$ is a polynomial over $\C(x_1,x_2)$ in the linear form 
$a_1(x_1,x_2)x_3+a_2(x_1,x_2)x_4+\cdots+a_{n-2}(x_1,x_2)x_n$ as well.

Take $d \ge 0$ arbitrary, and let $\tilde{f}$ be the part of $f$, which has
degree $d$ with respect to $x_3,x_4,\ldots,x_n$. Then $\tilde{f} \in \C[x]$,
and over $\C(x_1,x_2)$, $\tilde{f}$ is a monomial in the linear form 
$a_1(x_1,x_2)x_3+a_2(x_1,x_2)x_4+\cdots+a_{n-2}(x_1,x_2)x_n$. From 
Gauss's Lemma, it follows that $\tilde{f} \in A$, As $d$ was arbitrary,
we can conclude that $f \in A$. \qedhere

\end{enumerate}
\end{listproof}

The connection between quasi-translations and polynomial Hessians with determinant zero,
which comes from \cite{gornoet}, is given at the beginning of section \ref{hess}. 
This connection is used in \cite{garrep} and \cite[Ch.\@ 7]{russo} as well, and appears
as \cite[p.\@ 33]{garrep} and \cite[Lem.\@ 7.3.7]{russo} respectively.
\cite{garrep} and \cite[Ch.\@ 7]{russo} contain classifications in dimensions less than $5$ as well,
but with the same limitations as above on the factorization of $h$. These limitations are not 
present in \cite{gnlossen}, which follows the approach of \cite{gornoet} in proving the
classifications in dimensions less than $5$.

In \cite{watanabe}, it is claimed that $\rk \jac H \ne 3$ if $x + H$ is a quasi-translation
in dimension $n = 5$, but this is not true. Hence the proof in \cite{watanabe} of the classification 
of homogeneous polynomials in $5$ indeterminates, for which the Hessian determinant is zero,
has a gap. The paper \cite{franch} has an error and hence a gap on the same point.
This gap can be fixed by proving that $\rk \jac H \ne 3$ indeed, if $x + H$ 
is associated to a polynomial for which the Hessian determinant is zero, which can be
done by way of the results on linear invariants of quasi-translations, as given in
\cite{gornoet} and this paper: see remark \ref{rem} at the end of section \ref{hess}.

\cite{garrep} and \cite[Ch.\@ 7]{russo} on one hand, and \cite[Th.\@ 5.3.7]{homokema} 
on the other hand, treat the case where $\rk \jac H = 3$ incorrectly as well. But both 
incorrect treatments are only on subcases which do not overlap, so \cite[Ch.\@ 7]{russo}
and \cite[Th.\@ 5.3.7]{homokema} fix each other's errors.
The error in \cite{garrep} and \cite[Ch.\@ 7]{russo} can be repaired by way of theorem \ref{A}, 
which comes from \cite{gornoet}. The error in \cite[Th.\@ 5.3.7]{homokema} can be 
repaired by way of lemma \ref{B}, which gives a simpler argument than that in \cite{gornoet}.

It is easy to show that for any homogeneous polynomial map $H$ such that $\rk \jac H = 1$, $x + H$ has 
$n-1$ independent linear invariants. In \cite{gornoet}, Gordan and N{\"o}ther proved that any homogeneous 
quasi-translation $x + H$ such that $\rk \jac H = 2$ has at least $2$ independent linear invariants. 
In their study of homogeneous quasi-translations $x + H$ in dimension $n = 5$ with $\rk \jac H = 3$ in 
\cite{gornoet}, Gordan and N{\"o}ther distinguished two cases, namely `Fall a)' and `Fall b)', of which 
`Fall a)' had two subcases, which we indicate by a1) and a2).
 
The quasi-translations of subcase a1) in \cite{gornoet} are the homogeneous quasi-trans\-lations $x + H$ in
dimension $5$ with Jacobian rank three, for which the Zariski closure of the image of $H$ is a $3$-dimensional 
linear subspace of $\C^5$.
The quasi-translations of case b) in \cite{gornoet} are the homogeneous quasi-translations in
dimension $5$ with Jacobian rank three, which are linearly conjugate to another such quasi-translation 
$x + H$, for which $H_5$ is algebraically independent over $\C$ of $H_1, H_2, H_3, H_4$, but for which the
Zariski closure of the image of $H$ is not a $3$-dimensional linear subspace of $\C^5$.

The quasi-translations of subcase a2) in \cite{gornoet} are categorized by a somewhat technical property,
which is the existence of $p^{(1)}$ and $p^{(2)}$ as in (iii) of theorem \ref{Lpth}. Let us just say for
now that they are the homogeneous quasi-translations in dimension $5$ with Jacobian rank three, which do 
not belong to case b) or subcase a1) in \cite{gornoet}. 
As a consequence of theorem \ref{Lpth}, we deduce in corollary \ref{Lpcor} that quasi-translations of 
case a2) 
in \cite{gornoet} have at least one linear invariant, by showing that the linear span of the image of $H$ is 
$4$-dimensional. Having reasoned about these three cases, one can wonder whether they actually exist.
 
\begin{example} \label{a12b}
The following three $H$'s are chosen in such a way, that $x + H$ with $n = 5$ 
is a quasi-translation which belongs to the above-described case a1), a2), and b), 
respectively.
\begin{itemize}

\item[a1)] $H = (x_4^2,x_4x_5,x_1x_5-x_2x_4,0,0)$,

\item[a2)] $H = (x_5^2(ax_1-x_5^2x_2),a(ax_1-x_5^2x_2),x_5^2(ax_3-x_5^2x_4),
      a(ax_3-x_5^2x_4),0)$ with $a = x_1x_4-x_2x_3$,

\item[b)] $H = (x_5^5,bx_5^3,b^2x_5,-b^2x_1+2bx_2x_5^2-x_3x_5^4,0)$ with 
      $b = x_1x_3-x_2^2+x_4x_5$.

\end{itemize}
The quasi-translations for a1) and a2) were found by using techniques of \cite[\S 2]{debunk}.
The quasi-translations for b) was found by applying propositions \ref{qtconj} and
\ref{qthmg}, on the quasi-translation $x + H$ with $n = 4$ and $H = (1,x_4,x_4^2,0)$.
\end{example}

An unsolved question is whether a homogeneous quasi-translation in dimension $5$ always has
a linear invariant or not. We reprove the following results obtained in \cite{gornoet} in modern
language: a homogeneous quasi-translation in dimension $5$ without a linear invariant can only belong 
to case b) in \cite{gornoet}. Furthermore, we give a somewhat less computational proof of the 
result in \cite{liu} that a homogeneous quasi-translation in dimension $5$ without a linear 
invariant must have degree $15$ at least. 

In dimension $6$ and up, homogeneous 
quasi-translations do not need to have linear invariants, see \cite[Th.\@ 2.1]{debunk}.
If we substitute $x_5 = 1$ in the quasi-translations of cases a2) and b) in example \ref{a12b} and
remove the last component, we get non-homogeneous quasi-translations in dimension $4$ without
linear invariants.

The rest of the paper is organized as follows.
In the next section, we show some basic concepts about quasi-translations. 

In section \ref{imqt},
we prove some geometric results about homogeneous quasi-translations $x + H$ for which $\rk \jac H \le 
(n+1) / 2$. As a consequence, we deduce that a homogeneous quasi-translation in dimension $5$ without
linear invariants can only belong to case b) in \cite{gornoet}.

In section \ref{hess}, we apply the result that a homogeneous quasi-translation in dimension $5$ without a 
linear invariant can only belong to case b) in \cite{gornoet}, to classify all homogeneous polynomials in 
$5$ indeterminates for which the Hessian determinant vanishes. 

In section \ref{Fallb}, we study homogeneous 
quasi-translations in dimension $5$ that belong to case b) in \cite{gornoet}, with the purpose of getting 
properties of possible homogeneous quasi-translations in dimension $5$ without linear invariants. One of
these properties is that the degree of such a quasi-translation is at least $15$.

In section \ref{kerqt}, we prove some geometric results about quasi-translations which gives us the
following result about quasi-translations which belong to case b) in \cite{gornoet}: the Zariski 
closure of the image of $H$ is an irreducible component of $V(H)$, which contains a linear 
$1$-dimensional subspace $L$ of $\C^5$, such that every other irreducible component of 
$V(H)$ is a $3$-dimensional linear subspace of $\C^5$ which contains $L$. 
Here, $V(H)$ is the set of common zeroes of $H_1, H_2, \ldots, H_n$.

\section{Some basics about quasi-translations}

In proposition \ref{qtprop} below, we will show that quasi-translations are also
characterized by $H(x+tH) = H$ and by that $\jac H \cdot H$ is the zero vector. We need
the following lemma to prove proposition \ref{qtprop}.

\begin{lemma} \label{qtlem}
Assume that $x + H$ is a polynomial map and $f \in \C[x]$. Then
\begin{equation}
f(x+tH) = f(x) \label{fxtH}
\end{equation}
in case one of the following assumptions is satisfied.
\begin{enumerate}[\upshape (1)]

\item $x + H$ is a quasi-translation and $f(x + H) = f(x)$,

\item $\jac H \cdot H = (0^1,0^2,\ldots,0^n)$ and $\jac f \cdot H = 0$.

\end{enumerate}
\end{lemma}

\begin{listproof}
\begin{enumerate}[(1)]

\item Since $(x - H) \circ (x + H) = x$, we see that
\begin{align*}
(x + mH) \circ (x + H) &= \big((m+1)x - m(x - H)\big) \circ (x + H) \\
&= (m+1)(x + H) - mx = x + (m+1)H
\end{align*}
By induction on $m$, $x + mH$ is equal to the composition of $m$ copies of 
$x + H$ for all $m \in \N$. Using $f(x + H) = f(x)$ $m$ times, we obtain
$$
f(x + mH) = f\big((x+H)^{\circ m}\big) = f\big((x+H)^{\circ (m-1)}\big)
= \cdots = f(x)
$$
for all $m \in \N$. This is only possible if \eqref{fxtH} holds.

\item By the chain rule and $\jac H\cdot H = (0^1,0^2,\ldots,0^n)$, we get
\begin{align*}
\jac f(x + tH) \cdot H &= (\jac f)|_{x=x+tH} \cdot (I_n + t \jac H) \cdot H
\nonumber \\
&= (\jac f)|_{x=x+tH} \cdot H = \parder{}{t} f(x + tH)
\end{align*}
where $I_n$ is the unit matrix of size $n$. Since $\jac f \cdot H = 0$,
it follows from the above that 
\begin{equation} \label{qtinv}
\jac \big(f(x + tH) - f(x)\big) \cdot H = \parder{}{t} f(x + tH)
\end{equation}
Suppose that $t$ divides the right hand side of \eqref{qtinv} exactly $r < \infty$ times.
Then $t$ divides $f(x+tH) - f(x)$ more than $r$ times. Hence $t$ divides the left hand
side of \eqref{qtinv} more than $r$ times as well, which is a contradiction. 
So both sides of \eqref{qtinv} are zero. Since the right hand side of \eqref{qtinv} is zero, 
we get \eqref{fxtH}. \qedhere

\end{enumerate}
\end{listproof}

\begin{proposition} \label{qtprop}
Let $H: \C^n \rightarrow \C^n$ be a polynomial map. 
Then the following properties are equivalent:
\begin{enumerate}[\upshape (1)]

\item $x + H$ is a quasi-translation, 

\item $H(x + tH) = H$ (where $t$ is a variable),

\item $\jac H \cdot H = (0^1,0^2,\ldots,0^n)$.

\end{enumerate}
Furthermore, if any of {\upshape (1)}, {\upshape (2)} and {\upshape (3)} is satisfied, then
\begin{equation} \label{fxtHeqv}
f(x + H) = f(x) \Longleftrightarrow f(x + tH) = f(x) \Longleftrightarrow \jac f \cdot H = 0
\end{equation}
for all $f \in \C[x]$, and 
\begin{equation} \label{qtnilp}
(\jac H)|_{x=x-t\jac H} = (\jac H) + t (\jac H)^2 + t^2 (\jac H)^3 + \cdots
\end{equation}
\end{proposition}

\begin{proof}
The middle hand side of \eqref{fxtHeqv} gives the left hand side by substituting
$t = 1$ and the right hand side by taking the coefficient of $t^1$. Lemma \ref{qtlem}
gives the converse implications by way of (1) and (3).
Hence \eqref{fxtHeqv} follows as soon as we have the equivalence of (1), (2) and (3).

By taking the Jacobian of (2), we get $(\jac H)|_{x=x+tH} \cdot (I_n + t \jac H) 
= \jac H$, which gives \eqref{qtnilp} after substituting $t = -t$.
Therefore, it remains to show that (1), (2) and (3) are equivalent.
\begin{description}

\item [(1) \imp (2)]
Assume (1). Since $x = (x - H) \circ (x + H) = x + H - H(x+H)$, we see that
$H(x + H) = H$, and (2) follows by taking $f = H_i$ for each $i$ in (1) of lemma \ref{qtlem}.

\item [(2) \imp (1)]
Assume (2). Then
\begin{align} \label{ixtH}
(x - tH) \circ (x + tH) &= (x + tH) - tH(x + tH) \nonumber \\
&= x + tH - tH = x
\end{align}
which gives (1) after substituting $t = 1$. 

\item[(2) \imp (3)]
Assume (2). By taking the coefficient of $t^1$ of (2), we get (3).

\item[(3) \imp (2)]
Assume (3). By taking $f = H_i$ in (2) of lemma \ref{qtlem}, we get (2). \qedhere

\end{description}
\end{proof}

Proposition \ref{irred} below gives a tool to obtain quasi-translations
$x + H$ over $\C$ for which $\gcd\{H_1,H_2,\ldots,H_n\} = 1$ from arbitrary
quasi-translations $x + H$ over $\C$.

\begin{proposition} \label{irred}
Assume $x + gH$ is a quasi-translation over $\C$, where $g \in \C[x]$ is nonzero. 
Then $x + H$ is a quasi-translation over $\C$ as well. Furthermore, the invariants
of $x + H$ and $x + gH$ are the same. If additionally $H$ is homogeneous of positive degree, 
then $\rk \jac gH = \rk \jac H$.
\end{proposition}

\begin{proof}
By (1) $\Rightarrow$ (2) of proposition \ref{qtprop}, we see that
$g(x+tgH) \cdot H_i(x+tgH) = g \cdot H_i$. We can substitute $t = g^{-1}t$ in it,
to obtain that
\begin{align*}
\deg_t H_i(x+tH) &\le \deg_t g(x+tH) + \deg_t H_i(x+tH) \\ &= \deg_t (gH_i)(x+tH) \le 0
\end{align*}
for each $i$, which is exactly $H(x+tH) = H$. Hence $x + H$ is 
a quasi-translation on account of (2) $\Rightarrow$ (1) of 
proposition \ref{qtprop}.

Assume $f$ is an invariant of $x + H$. Then $f(x + tH) = f(x)$ on account of 
\eqref{fxtHeqv}, and by substituting $t = g$ we see that $f$ is an invariant of 
$x + gH$. The converse follows in a similar manner by substituting $t = g^{-1}$.

Suppose that $H$ is homogeneous of positive degree. From Proposition 1.2.9 of either \cite{arnobook} 
or \cite{homokema}, we deduce that in order to prove that $\rk \jac gH = \rk \jac H$, 
it suffices to show that $\trdeg_{\C} \C(gH) = \trdeg_{\C} \C(H)$.
For that purpose, we show that for any $R \in \C[y]$, both $R(gH)$ and $R(H)$ are zero if one of them is.

Suppose that either $R(gH) = 0$ or $R(H) = 0$ for some $R \in \C[y]$, say of degree $r$. Let $\bar{R}$
be the leading homogeneous part of $R$. If $R(H) = 0$, then $\bar{R}(H) = 0$ because $H$ is homogeneous
of positive degree. If $R(gH) = 0$, then $\deg \bar{R}(gH) < r \deg gH = 
\deg g^r + r \deg H$, so $\deg \bar{R}(H) < r \deg H$, which is only possible if $\bar{R}(H) = 0$.
So $\bar{R}(gH) = \bar{R}(H) = 0$ in any case. Hence either $(R-\bar{R})(gH) = 0$ or $(R-\bar{R})(gH) = 0$. 
By induction to the number of homogeneous parts of $R$, it follows that $R(gH) = R(H) = 0$ indeed.
\end{proof}

Proposition \ref{qtconj} gives a criterion about preservation of the quasi-translation property with respect
to conjugation with an invertible polynomial map.

\begin{proposition} \label{qtconj}
Assume $x + H$ is a quasi-translation in dimension $n$ over $\C$, and 
$F$ is an invertible polynomial map in dimension $n$ over $\C$
with inverse $G$. Then
$$
G \circ (x + H) \circ F
$$
is a quasi-translation as well, if and only if $\deg_t G_i(x+tH) \le 1$ for all $i$.
In particular, if $T$ is an invertible matrix of size $n$ over $\C$, we have that
$$
x + T^{-1} H(Tx) = T^{-1} \big(Tx + H(Tx)\big) = T^{-1}x \circ (x + H) \circ Tx
$$
is a quasi-translation as well.
\end{proposition}

\begin{proof}
Assume first that $\deg_t G_i(x+tH) \le 1$ for all $i$. Then we can write
$$
G(x+tH) = G^{(0)} + tG^{(1)}
$$
Notice that $G^{(0)} = G(x+tH)|_{t=0} = G$. Hence
$$
G \circ (x+tH) \circ F = G^{(0)}(F) + t G^{(1)}(F) = G(F) + t G^{(1)}(F) = x + t G^{(1)}(F)
$$
By substituting $t = 1$ on both sides, we obtain that $G \circ (x + H) \circ F
= x + G^{(1)}(F)$ and substituting $t = -1$ tells us that its inverse
$G \circ (x - H) \circ F$ is equal to $x - G^{(1)}(F)$. Thus $G \circ (x + H) \circ F$
is a quasi-translation indeed.

Assume next that $G \circ (x + H) \circ F$ is a quasi-translation $x + \tilde{H}$. Then
$x - \tilde{H}$ is the inverse of $G \circ (x + H) \circ F$, which is $G \circ (x - H) \circ F$.
Hence
$$
\tilde{H} = \big(G \circ (x + H) \circ F\big) - x = x - \big(G \circ (x - H) \circ F\big)
$$
Substituting $x = G(x + mH)$ in the above gives
$$
G\big(x + mH + H(x + mH)\big) - G(x + mH) = G(x + mH) - G\big(x + mH - H(x + mH)\big)
$$
Since $H(x + mH) = H$ on account of (1) $\Rightarrow$ (2) of proposition \ref{qtprop}, we obtain
$$
G(x + (m+1)H) - G(x + mH) = G(x + mH) - G(x + (m-1)H)
$$
By induction on $m$, we get
$G(x + (m+1)H) - G(x + mH) = G(x + H) - G(x)$ for all $m \in \N$, whence
$$
G(x + \tilde{m}H) - G(x) = \sum_{m=0}^{\tilde{m}-1} G(x + (m+1)H) - G(x + mH) 
= \tilde{m} (G(x + H) - G(x))
$$
for all $\tilde{m} \in \N$. This is only possible if $G(x + t H) - G(x) = t (G(x + H) - G(x))$. 
Hence $\deg_t G(x + t H) \le 1$, as desired.
\end{proof}

Proposition \ref{qthmg} gives a tool to obtain homogeneous quasi-translations 
over $\C$ from arbitrary quasi-translations $x + H$ over $\C$. Hence we can
obtain results about arbitrary quasi-translations by studying homogeneous ones.

\begin{proposition} \label{qthmg}
Assume $x + H$ is a quasi-translation over $\C$ in dimension $n$, and 
$$
d \ge \deg H := \max\{\deg H_1, \deg H_2, \ldots, \deg H_n\}
$$
Then 
$$
(x,x_{n+1}) + x_{n+1}^d \big(H(x_{n+1}^{-1}x), 0\big)
$$
is a {\em homogeneous} quasi-translation over $\C$ in dimension $n + 1$.
\end{proposition}

\begin{proof}
Denote 
$$
(x,x_{n+1}) =: \tilde{x} \qquad \mbox{and} \qquad 
x_{n+1}^d \big(H(x_{n+1}^{-1}x), 0\big) =: \tilde{H}
$$
We must show that $\tilde{x} + \tilde{H}$ is a quasi-translation 
in dimension $n+1$ over $\C$. On account of (3) $\Rightarrow$ (1) of 
proposition \ref{qtprop}, it suffices to show that 
$\jac_{\tilde{x}} \tilde{H} \cdot \tilde{H} = (0^1,0^2,\ldots,0^{n+1})$. 
Since $\tilde{H}_{n+1} = 0$, this is equivalent to 
$$
\jac \tilde{H} \cdot x_{n+1}^d H(x_{n+1}^{-1}x) = (0^1,0^2,\ldots,0^{n+1})
$$
Using that $\jac \tilde{H}_{n+1}$ is the zero row, we see that it suffices to show that 
$$
\jac \big(x_{n+1}^d H(x_{n+1}^{-1}x)\big) \cdot x_{n+1}^d H(x_{n+1}^{-1}x) = (0^1,0^2,\ldots,0^n)
$$
This is indeed the case, because the chain rule tells us that
\begin{align*}
 (0^1,0^2,\ldots,0^n) &= x_{n+1}^{2d-1} \cdot (0^1,0^2,\ldots,0^n) \\
  &= x_{n+1}^{2d-1} \cdot (\jac H \cdot H)_{x = x_{n+1}^{-1}x} \\
  &= x_{n+1}^{2d-1} \cdot (\jac H \cdot x_{n+1}^{-1} \cdot x_{n+1} H)_{x = x_{n+1}^{-1}x} \\
  &= x_{n+1}^{2d-1} \cdot \jac \big(H(x_{n+1}^{-1}x)\big) \cdot x_{n+1} H(x_{n+1}^{-1}x) \\
  &= \jac \big(x_{n+1}^d H(x_{n+1}^{-1}x)\big) \cdot x_{n+1}^d H(x_{n+1}^{-1}x) \qedhere
\end{align*}
\end{proof}

Proposition \ref{hmgprop} below connects quasi-translations with homogeneity.

\begin{proposition} \label{hmgprop}
Assume $H$ is a homogeneous polynomial map over $\C$. Then the assertions
\begin{enumerate}[\upshape (1)]

\item $\jac H^2$ is the zero matrix,

\item $x + H$ is a quasi-translation,

\item $H(H) = (0^1,0^2,\ldots,0^n)$ and $\rk \jac H \le \max\{n-2,1\}$,

\end{enumerate}
satisfy {\upshape (1)} $\Rightarrow$ {\upshape (2)} $\Rightarrow$  {\upshape (3)}.
\end{proposition}

\begin{proof}
Suppose that $H$ is homogeneous of degree $d$. Let $E: \C[x]^n \rightarrow \C[x]^n$ be the map which 
multiplies each term in any of the $n$ components by its own degree. Then one can verify that 
$E(H) = \jac H \cdot x$. So $\jac H \cdot H = d^{-1} \jac H \cdot d H = d^{-1} \jac H \cdot E(H) = 
d^{-1} \jac H^2 \cdot x$. Hence (1) $\Rightarrow$ (2) follows from (3) $\Rightarrow$ (1) of proposition 
\ref{qtprop}.

In order to prove (2) $\Rightarrow$ (3), assume that (2) holds. 
By looking at the coefficient of $t^d$ of $H(x+tH) - H(x)$, we deduce that 
$H(H) = (0^1,0^2,\ldots,0^n)$, which is the first claim of (3). 

To show the second claim of (3), assume that $\rk \jac H > 1$. Write $H = g \tilde{H}$, 
where $g \in \C[x]$, such that $\gcd\{\tilde{H}_1,\tilde{H}_2,\ldots,\tilde{H}_n\} = 1$. 
Since $\rk \jac H  > 1$, we have $\deg \tilde{H} \ge 1$. Furthermore, $V(\tilde{H})$ cannot be written as
a zero set of a single polynomial. Since $\C[x]$ is a unique factorization domain,
we see that $\dim V(\tilde{H}) \le n-2$. 

Using proposition \ref{irred}, Proposition 1.2.9 of either \cite{arnobook} or \cite{homokema}, 
and the above obtained $\tilde{H}(\tilde{H}) = 0$ and $\dim V(\tilde{H}) \le n-2$, in that order, 
we deduce that
$$
\rk \jac H = \rk \jac \tilde{H} = \trdeg_{\C} \C(\tilde{H}) \le \dim V(\tilde{H}) \le n-2 
$$
which gives the second claim of (3).
\end{proof}

\mathversion{bold}
\section{The image of the map $H$ of quasi-translations $x + H$} \label{imqt}
\mathversion{normal}

We prove several results about quasi-translations with geometrical arguments.
Some of these results have been claimed by Paul Gordan and Max N{\"o}ther in \cite{gornoet}. 
For the last two sections, we need several parts of corollary \ref{Lpcor} in this section.

Since the results may essentially be useful for non-homogeneous quasi-trans\-lations as well,
it does not seem to be a good idea to work with projective varieties. But we will need
the completeness of complex projective space in some manner. The lemma below gives us an
affine version of that.

\begin{lemma} \label{completeness}
Let $\tilde{Z} \subseteq \C^{m+kn}$ be closed with respect to the Euclidian topology. Assume that 
for every point of $\tilde{Z}$, the projection onto its last $kn$ coordinates gives a point of 
$\C^{kn}$ with complex norm $\sqrt{k}$. Let $\tilde{X}$ be the image of the projection of $\tilde{Z}$ 
onto its first $m$ coordinates.

Suppose that there is an irreducible variety $X \subseteq \C^{m}$ and a Zariksi open set 
$U$ of $X$, such that $U \subseteq \tilde{X} \subseteq X$. Then $\tilde{X} = X$.
\end{lemma}

\begin{proof}
Since the set of points in $\C^{kn}$ whose complex norm is $\sqrt{k}$ form a compact space, the projection
of $\tilde{Z}$ onto $\tilde{X}$ is closed with respect to the Euclidean topology. 
Hence $\tilde{X}$ is closed in the Euclidean topology. So $\tilde{X}$ contains the Euclidean closure of 
$U$ in $X$. On account of \cite[Th.\@ 7.5.1]{advanced}, the Euclidean closure of $U$ in $X$ is the same as the 
Zariksi closure of $U$ in $X$, which is $X$. Hence $X \subseteq \tilde{X}$. So $\tilde{X} = X$ indeed.
\end{proof}

Notice that reverting to Euclidean topology is not only because the complex inner product cannot be
expressed as a polynomial, but also because the Zariski topology of a product is not the corresponding 
product topology.

We also need a weak form of the projective fiber dimension theorem in some manner. 
Lemma \ref{projfiber} below is an affine version of that. But first, we need another lemma.

\begin{lemma} \label{Wrk}
Suppose that $H \in \C[x]^n$. Then the Zariski closure $W$ of the image of $H$
is irreducible and has dimension $\rk \jac H$. 

Furthermore, $V(H)$ has dimension at least $n - \rk \jac H$ if $H$ has no constant part.
\end{lemma}

\begin{proof}
From Proposition 1.2.9 of either \cite{arnobook} or \cite{homokema}, it follows that
$\rk \jac H = \trdeg_{\C} \C(H)$. Hence $\dim W = \rk \jac H$ indeed. 

Let $Z$ be a component of $W$ and let $Y$ be the union of the other components of $W$.
By definition of $Z$, $U := H^{-1}(W \setminus Y) \ne \varnothing$. 
By continuity of $H$, $U$ is open and $H^{-1}(Z) \supseteq U$ is closed,
so $H^{-1}(Z) = \C^n$ and $W = Z$ is irreducible.

To prove the last claim, suppose that $H$ has no constant part. Then $0 \in V(H)$. 
From a weak version of the affine fiber dimension theorem 
(or from lemma \ref{projfiber} below, applied on the map $(H,x_{n+1})$), 
it follows that $\dim V(H) = \dim H^{-1}(0) \ge n - \rk \jac H$ indeed.
\end{proof}

\begin{lemma} \label{projfiber}
Suppose that $H: \C^n \rightarrow \C^n$ is a polynomial map and $p \in \C^n$, such 
that the linear span $\C p$ of $p$ contains infinitely many points of the image of $H$.
Then there exists an irreducible component $X$ of $H^{-1}(\C p)$ such that $H(X)$ has
infinitely many points, and the dimension of any such $X$ is larger than $n - \rk \jac H$. 
\end{lemma}

\begin{proof}
Let $W$ be the Zariski closure of the image of $H$. On account of lemma \ref{Wrk}, 
$\dim W = \rk \jac H$. Take a generic linear subspace $L \ni p$ of dimension 
$n + 1 - \rk \jac H$ of $\C^n$, so that $\dim (L \cap W) = 1$. The set $Y := \{c \in \C^n \mid H(c) \in L\}$ 
is the zero set of $\rk \jac H - 1$ $\C$-linear forms in the components of $H$.
By applying \cite[Ch.\@ I, Prop.\@ 7.1]{hartshorne} $\rk \jac H - 2$ times, it follows that
every irreducible component of $Y$ have dimension greater than $n - \rk \jac H$. 
Furthermore, $\dim H(Y) = 1$ because $H(Y) = L \cap W$.

Since $\C p \cap H(Y)$ contains infinitely many points and $Y$ has finitely many irreducible components, 
there is an irreducible component $X$ of $Y$ such that $H(X)$ has infinitely many points of $\C p$. 
Furthermore, $\dim X > n - \rk \jac H$, because all irreducible components of $Y$ have dimension 
greater than $n - \rk \jac H$. So it remains to show that $X \subseteq H^{-1}(\C p)$.

Since $H(X)$ has infinitely many points of $\C p$, it follows that $\C p$ is contained in the 
Zariski closure of $H(X)$. As $\dim H(X) \le \dim H(Y) = 1 = \dim \C p$, $\C p$ is 
a component of the Zariski closure of $H(X)$. Now $X \subseteq H^{-1} (\C p)$ follows
in a similar manner as $\C^n \subseteq H^{-1}(Z)$ in the proof of lemma \ref{Wrk}.
\end{proof}

\begin{lemma} \label{pqfiber}
Assume $x + H$ is a homogeneous quasi-translation over $\C$. Suppose that $p$ and $q$ 
are independent and contained in the image of $H$. Then there exists an algebraic set $X$
of dimension at least $n - 2(\rk \jac H - 1)$, such that $H(c + tp) = H(c + tq) = 0$ for all 
$c \in X$.
\end{lemma}

\begin{proof}
On account of lemma \ref{projfiber}, there exist irreducible algebraic sets $X_p$ and $X_q$
of dimension at least $n + 1 - \rk \jac H$, such that $H(X_p)$ and $H(X_q)$ contain infinitely
many points of $\C p$ and $\C q$ respectively. The set $X_p \cap H^{-1}(\C^{*}p)$ is an open 
subset of $X_p$, and its Zariski closure is just $X_p$ because $X_p$ is irreducible.
For $c \in H^{-1}(\C^{*}p)$, we have $H(c + t p) = H(c) = \lambda p$ for some 
$\lambda \in \C$ on account of (1) $\Rightarrow$ (2) of proposition \ref{qtprop}. 
Hence $H(c + t p) = H(c) \in \C p$ for every $c \in X_p$. 

By a similar argument with $q$ instead of $p$, we see that $H(c + t p) = H(c) = H(c + t q)$ 
is dependent of both $p$ and $q$ for every $c \in X_p \cap X_q$.
Due to the homogeneity of $H$, $0 \in X_p \cap X_q$. Hence it follows from 
\cite[Ch.\@ I, Prop.\@ 7.1]{hartshorne} that the dimension of $X_p \cap X_q$ is at least
$n - 2(\rk \jac H - 1)$. So $X = X_p \cap X_q$ suffices.
\end{proof}

\begin{lemma} \label{hmgqt5lm}
Assume $x + H$ is a homogeneous quasi-translation in dimension $n \le 5$ over $\C$, such that 
$\rk \jac H = 2$ and $\dim V(H) \le n-2$. Then $V(H)$ contains the linear span of the 
image of $H$.
\end{lemma}

\begin{proof}
$V(H)$ contains only finitely many $(n-2)$-dimensional linear subspaces of $\C^n$ because 
$\dim V(H) \le n-2$. Furthermore, the Zariski closure of the image of $H$ is irreducible
on account of lemma \ref{Wrk}.
From those two facts, we can deduce that it suffices to show that every nonzero $p$ in the image of 
$H$ is contained in an $(n-2)$-dimensional linear subspace of $\C^n$ which is contained in $V(H)$.

So take any nonzero $p$ in the image of $H$. Take $q$ independent of $p$ such that $q$ is 
the image of $H$ as well. From lemma \ref{pqfiber}, it follows that there exists an algebraic
set $X$ of dimension at least $n - 2(\rk \jac H - 1) = n - 2$, such that $H(c + tp) = H(c + tq) = 0$
for all $c \in X$. Choose $X$ irreducible. Since $\dim V(H) \le n - 2$ and $X \subseteq V(H)$, 
it follows that $\dim X = n-2$ and that the interior $X^{\circ}$ of $X$ as a closed 
subset of $V(H)$ is nonempty. 

Take $c \in X^{\circ}$, such that $c$ is independent of $p$ and $q$ if $n = 5$.
Then the linear span of $c$, $p$ and $q$ has dimension at least $\max\{2,n-2\}$.
Since $H(c + tp) = 0$, the linear span $L$ of $c$ and $p$ is contained in $V(H)$.
Since $c \in L \subseteq V(H)$ and $c \in X^{\circ}$, it follows from the irreducibility of $L$ that
$L \subseteq X$. 

In as similar manner, it follows that for every $\tilde{c} \in L \cap X^{\circ}$,
hence for all $\tilde{c} \in L$, the linear span of $\tilde{c}$ and $q$ is contained in $V(H)$.
So the linear span of $L$ and $q$ is contained in $V(H)$. 
This linear span has dimension at least $\max\{2,n-2\}$.
Since $\dim V(H) \le n-2$, it follows that $n \ge 4$ and that $p$ is contained 
in an $(n-2)$-dimensional linear subspace of $\C^n$ which is contained in $V(H)$.
\end{proof}

\begin{theorem}[Gordan and N{\"o}ther] \label{hmgqt5th}
Assume $x + H$ is a homogeneous quasi-translation over $\C$, such that 
$\deg H \ge 1$.
\begin{enumerate}[\upshape (i)]

\item If $\rk \jac H \le 1$, then the image of $H$
is a line through the origin and $x + H$ has $n-1$ independent linear invariants.

\item If $\gcd\{H_1,H_2,\ldots,H_n\} = 1$, then 
$2 \le \rk \jac H \le \dim V(H) \le n-2$.

\item If $\rk \jac H = 2$, then $x + H$ has at least two independent 
linear invariants.

\end{enumerate}
\end{theorem}

\begin{proof}
For the moment, we prove (iii) only for the case where $n \le 5$, because we do not need
the case where $n \ge 6$ in this paper. To prove the general case of (iii), one can replace the 
use of lemma \ref{hmgqt5lm} by that of the more general corollary \ref{hmgrk2cor}
in the last section.

Let $W$ be the Zariski closure of the image of $H$. From lemma \ref{Wrk}, it follows that
$W$ is irreducible and that $\dim W = \rk \jac H$.
\begin{enumerate}[(i)]
 
\item As $\deg H \ge 1$, the case $\rk \jac H = 0$ is impossible. So assume that $\rk \jac H = 1$.
Since $H$ is homogeneous and $\dim W = \rk \jac H = 1$, it follows from the irreducibility 
of $W$ that the image of $H$ can only be a line through the origin.
Hence there are $n-1$ independent linear forms $l_1, l_2, \ldots, l_{n-1}$ which vanish on the 
image of $H$. So $l_1, l_2, \ldots, l_{n-1}$ are invariants of $x + H$.

\item Assume that $\gcd\{H_1,H_2,\ldots,H_n\} = 1$. Since $\deg H \ge 1$, it follows from
(i) that $\rk \jac H \ge 2$. From (2) $\Rightarrow$ (3) of proposition \ref{hmgprop}, it follows
that $\rk \jac H  \le n - 2$, but its proof tells us that even $\rk \jac H \le \dim V(H) \le n - 2$.
So $2 \le \rk \jac H \le \dim V(H) \le n - 2$.

\item Assume that $\rk \jac H = 2$. From lemma \ref{Wrk}, it follows that 
$\dim V(H) \ge n-\rk \jac H = n-2$. Write $H = g \tilde{H}$, where $g \in \C[x]$, such that 
$\gcd\{\tilde{H}_1,\tilde{H}_2,\ldots,\allowbreak  \tilde{H}_n\} = 1$. Since $\rk \jac H = 2 > 1$, 
we have $\deg \tilde{H} \ge 1$. On account of proposition \ref{irred}, $\rk \jac \tilde{H} = \rk \jac H = 2$.
Furthermore, $2 \le \dim V(\tilde{H}) \le n-2$ on account of (ii), so $n \ge 4$.

From lemma \ref{hmgqt5lm}, it follows that the linear span of the image of $\tilde{H}$ 
is contained in $V(\tilde{H})$. Since $\dim V(\tilde{H}) \le n-2$, the linear span of the image of 
$\tilde{H}$ has dimension at most $n-2$ as well. Hence there are at least two independent linear forms 
$l_1$ and $l_2$ which vanish on the image of $\tilde{H}$. Thus $l_i(\tilde{H}) = 0$ and 
$l_i(H) = g \cdot 0 = 0$ for both $i \le 2$. So $l_1$ and $l_2$ are invariants of $x + H$. \qedhere

\end{enumerate}
\end{proof}

\begin{definition}
Let $H$ be a polynomial map. We define a \emph{GN-plane} of $H$ as a $2$-dimensional linear subspace
of $\C^n$ which is contained in $V(H)$. 
\end{definition}

\begin{theorem} \label{Lpth}
Assume $x + H$ is a homogeneous quasi-translation over $\C$, such that $2 \le \rk \jac H 
\le (n+1)/2$. Write $W$ for the Zariski closure of the image of $H$.
\begin{enumerate}[\upshape (i)]
 
\item For each $p \in W$ and each $q \in W$, there are GN-planes $L_p \ni p$ and 
$L_q \ni q$ of $H$ which intersect nontrivially.

\item If there exists a $p \in W$ which is contained in only finitely many GN-planes 
of $H$, then the set of such $p \in W$ is not contained in a proper algebraic subset of $W$.

\item Suppose that $p^{(1)}, p^{(2)}, \ldots, p^{(k)} \in W$, such that $p^{(i)}$ is contained
in only finitely many GN-planes of $H$ for each $i$. \\
Then there exist GN-planes 
$L_{p^{(1)}} \ni p^{(1)}$, $L_{p^{(2)}} \ni p^{(2)}$, \ldots, $L_{p^{(k)}} \ni p^{(k)}$ of $H$, 
such that for each $q \in W$, there exists a GN-plane $L_q \ni q$ of $H$ which intersects
$L_{p^{(i)}}$ nontrivially for each $i$.

\end{enumerate}
\end{theorem}
 
\begin{listproof}
\begin{enumerate}[(i)] 
 
\item
We first show that (i) holds for all $(p,q)$ in a dense open subset of $W^2$. The generic
property of $p$ and $q$ that we assume is that $p$ and $q$ are independent and contained in the 
image of $H$ itself. From \cite[\S 1.8, Th.\@ 3]{redbook}, it follows that the image of $H$
contains an open subset of $W$, so that we can easily show that we are considering a dense open 
subset of $W^2$ indeed. From lemma \ref{pqfiber}, it follows that there exists an algebraic set $X$
of dimension at least $n - 2(\rk \jac H - 1) \ge 1$, such that $H(c + t p) = H(c + t q) = 0$
for every $c \in X$. Take $c \in X$ nonzero. Since $H$ is homogeneous, we deduce by substituting 
$t = t^{-1}$ that $H(tc + p) = H(tc + q) = 0$.

In the general case, consider the sets 
$$
Z := \{(p,q,c,b) \in W^2 \times (\C^n)^2 \mid H(tc + p) = H(tc + q) = 0 \mbox{ and } b\tp c = 1 \}
$$
and 
$$
\tilde{Z} := \{(p,q,c,b) \in Z \mid b \mbox{ is the complex conjugate of } c\}
$$
By applying proper substitutions in $t$, we see that the image $\tilde{X}$ of the projection of 
$\tilde{Z}$ onto its first $2n$ coordinates is equal to that of $Z$.
Since $\tilde{X}$ contains an open subset of $X := W \times W$, it follows from lemma \ref{completeness}
that $\tilde{X} = X$, which gives (i).

\item 
Suppose that there exists a $p \in W$ for which there are only finitely many GN-planes $L_p \ni p$.
Let $Y$ be the set of $q \in W$ for which there are infinitely many GN-planes $L_q \ni q$. It is clear
that (ii) holds if $Y = \{0\}$, so assume that there exist a $q \in Y$ which is nonzero.
Take $P := \{c \in V(H) \mid H(c + tp) = 0\}$ and $Q := \{c \in V(H) \mid H(c + tq) = 0\}$. Since
$H$ is homogeneous, we see that both $P$ and $Q$ are unions of GN-planes. Furthermore,
$\dim P = 2$ and $\dim Q \ge 3$ because of the cardinality assumptions on the GN-planes in $P$ and $Q$.

Let $L$ be a generic linear subspace of dimension $n-2$ of $\C^n$, so that $\dim (L \cap P) = 0$. 
Then $L \cap P = \{0\} \subseteq L \cap Q$ and on account of 
\cite[Ch.\@ I, Prop.\@ 7.1]{hartshorne}, $\dim (L \cap Q) \ge 1$. Now define
$$
Z := \{(r,c,b) \in W \times L \times \C^n \mid H(tc + r) = 0 \mbox{ and }  b\tp c = 1 \}
$$
and 
$$
\tilde{Z} := \{(r,c,b) \in Z \mid b \mbox{ is the complex conjugate of } c\}
$$
By applying proper substitutions in $t$, we see that the image $\tilde{X}$ of the projection of 
$\tilde{Z}$ onto its first $n$ coordinates is equal to that of $Z$.
Furthermore $q$ is contained in $\tilde{X}$, but $p$ is not. Since $q \in Y \setminus \{0\}$ was arbitrary, 
we see that $Y \subseteq \tilde{X}$.

If $Y$ would contain an open subset of $W$, then lemma \ref{completeness} tells us that
$\tilde{X} = W$, which contradicts that $p$ is not contained in $\tilde{X}$.
So $Y$ does not contain an open subset of $W$, and $W \setminus Y$ is not contained in a proper
closed subset of $W$ indeed.

\item
We can simplify (iii) by changing both the quantization set of $q$ and the quantization order, to get the 
following. 
\begin{enumerate}
\item[(iii$'$)]
Suppose that $p^{(1)}, p^{(2)}, \ldots, p^{(k)} \in W$, such that $p^{(i)}$ is contained
in only finitely many GN-planes of $H$ for each $i$. \\
Then for each $q \in W$ which contains only finitely many GN-planes of $H$, there exist
a GN-plane $L_q$ of $H$ and GN-planes $L_{p^{(1)}} \ni p^{(1)}$, $L_{p^{(2)}} \ni p^{(2)}$, \ldots, 
$L_{p^{(k)}} \ni p^{(k)}$ of $H$, such that $L_q$ and $L_{p^{(i)}}$ intersect nontrivially for each $i$.
\end{enumerate}
The case where $k = 1$ of this simplification follows from (i). The case where $k \ge 2$ of this
simplification follows from the case where $k = 1$ of the unsimplified (iii) with $p^{(1)} = q$,
which may be assumed by induction on $k$.

So it remains to deduce (iii) from its simplification. For that purpose, define $Y$ as
\begin{equation*}
\begin{split}
Y := \big\{& (q,c^{(1)},c^{(2)},\ldots,c^{(k)},b^{(1)},b^{(2)},\ldots,b^{(k)}) \in
             W \times (\C^n)^{2k} ~\big| \\
     & H(tc^{(i)} + q) = H(tc^{(i)} + p^{(i)}) = 0 \mbox{ and } 
       (b^{(i)})\tp c^{(i)} = 1 \\
     & \mbox{ for each } i, \mbox{ and } 
       \rk \big(\,q\,\big|\,c^{(1)}\,\big|\,c^{(2)}\,\big|\,\cdots\,\big|\,c^{(k)}\,\big) 
       \le 2 \big\}
\end{split}
\end{equation*}
We can write $Y$ as a union of algebraic sets of the form
\begin{equation} \label{uniform}
\begin{split}
\big\{& (q,c^{(1)},\ldots,c^{(k)},b^{(1)},\ldots,b^{(k)}) \in Y ~\big|~
            c^{(i)} \in L_{p^{(i)}} \\
      & \mbox{ for each } i, \mbox{ and } 
        \rk \big(\,q\,\big|\,c^{(1)}\,\big|\,c^{(2)}\,\big|\,\cdots\,\big|\,c^{(k)}\,\big) 
        \le 2 \big\}
\end{split}
\end{equation}
where $L_{p^{(i)}} \ni p^{(i)}$ is a GN-plane of $H$ for each $i$. This union is finite by assumption.

Let $f$ be the projection of $\C^{n+2kn}$ onto its first $n$ coordinates. From the simplified version of (iii), 
it follows that the image of $f|_{Y}$ contains all $q \in W$ which contains only finitely many GN-planes 
of $H$. Om account of (ii), the image of $f|_{Y}$ is not contained in a proper algebraic subset of $W$. Hence
there exists an irreducible component $Z$ of $Y$ such that the image of $f|_{Z}$ is not contained in a
proper algebraic subset of $W$. From \cite[\S 1.8, Th.\@ 3]{redbook}, it follows that the image of 
$f|_{Z}$ contains an open subset of $W$. 

Since $Y$ is a finite union of algebraic subsets of the form \eqref{uniform} and $Z$ is irreducible, 
we deduce that $Z$ is contained in an algebraic subset of the form \eqref{uniform}. Take
\begin{equation*}
\begin{split}
\tilde{Z} := \big\{&(q,c^{(1)},c^{(2)},\ldots,c^{(k)},b^{(1)},b^{(2)},\ldots,b^{(k)}) \in Z ~\big| \\
                   &b^{(i)} \mbox{ is the complex conjugate of } c^{(i)} \mbox{ for each } i\big\}
\end{split}
\end{equation*}
By applying proper substitutions in $t$ and $y_1,y_2,\ldots,y_k$, we see that the image 
$\tilde{X}$ of $f|_{\tilde{Z}}$
is the same as that of $f|_{Z}$, so $\tilde{X}$ contains an open subset of $W$. From lemma 
\ref{completeness}, it follows that $\tilde{X} = W$. Since $\tilde{X}$ is the image of the 
restriction of $f$ on an algebraic subset of the form \eqref{uniform}, the unsimplified (iii)
follows.
\qedhere

\end{enumerate}
\end{listproof}

\begin{definition}
Let $X$ be any subset of $\C^n$. We say that $a \in \C^n$ is an \emph{apex} of $X$ 
if $(1-\lambda)c + \lambda a \in X$ for all $\lambda \in \C$ and all $c \in X$. 

We say that a $p \in \C^n$ is a \emph{projective apex} of $X$ if $p \ne 0$ and
$c + \lambda p \in X$ for all $\lambda \in \C$ and all $c \in X$. 

If $X$ is the Zariski closure of the image of a map $H$, then we say that $a$ and $p$ 
as above are an \emph{image apex} of $H$ and a \emph{projective image apex} of $H$ respectively.
\end{definition}

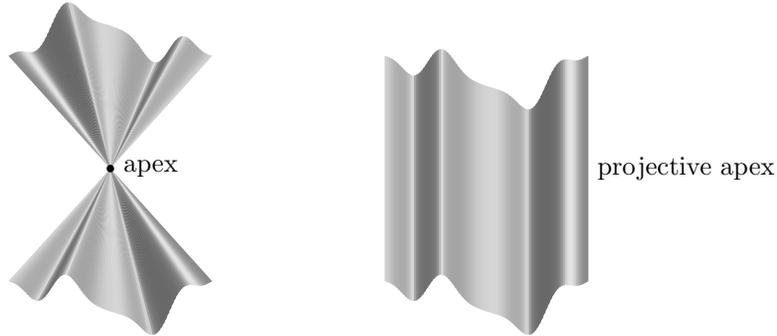
\begin{center}
\begin{tikzpicture}
\foreach \ang in {0.5,1.5,...,90} {
  \pgfmathsetmacro{\yl}{-2*sin(\ang*2-1)+sin(\ang*4-2)-0.7*sin(\ang*8-4)-
                        0.5*sin(\ang*12-6)+0.3*sin(\ang*16-8)+0.2*sin(\ang*20-10)}
  \pgfmathsetmacro{\yr}{-2*sin(\ang*2+1)+sin(\ang*4+2)-0.7*sin(\ang*8+4)-
                        0.5*sin(\ang*12+6)+0.3*sin(\ang*16+8)+0.2*sin(\ang*20+10)}
  \pgfmathsetmacro{\x}{-cos(\ang*2)+cos(\ang*4)-1.4*cos(\ang*8)-
                       1.5*cos(\ang*12)+1.2*cos(\ang*16)+cos(\ang*20)}
  \pgfmathsetmacro{\c}{10+0.5*abs((atan(\x)+20))};
  \fill[black!\c] (0.03*\ang-0.02,0.2*\yl) -- (0.03*\ang+0.02,0.2*\yr) -- 
                  (2.7-0.03*\ang-0.02,3-0.2*\yr) -- (2.7-0.03*\ang+0.02,3-0.2*\yl) -- cycle;
  \fill[black!\c] (0.03*\ang-0.02+5,0.2*\yl) -- (0.03*\ang+0.02+5,0.2*\yr) -- 
                  (0.03*\ang+0.02+5,3+0.2*\yr) -- (0.03*\ang-0.02+5,3+0.2*\yl) -- cycle;
}
\fill (1.35,1.5) circle (0.5mm);
\node[anchor=west] at (1.4,1.5) {apex};
\node[anchor=west] at (7.7,1.5) {projective apex};
\end{tikzpicture}
\end{center}

One may convince oneself that a projective apex is in fact an apex on the projective horizon.

If $X$ is a zero set of homogeneous polynomials, e.g.\@ because $X$ is the Zariski closure of the image
of a homogeneous map, then $0$ is an apex of $X$. If $0$ is an apex of $X$, then a projective
image apex is the same as a nonzero apex. In that case, we will parenthesize the word projective.

\begin{corollary} \label{Lpcor}
Assume $x + H$ is a homogeneous quasi-translation over $\C$, such that $\rk \jac H \le (n+1)/2$. 
Write $W$ for the Zariski closure of the image of $H$. Then for
\begin{enumerate}[\upshape (1)]
 
\item $\dim V(H) = \rk \jac H \le 3$ and $W$ has no nonzero (projective) apex;

\item $\dim V(H) = \rk \jac H$ and there is no nonzero $p \in W$ which contains infinitely many 
GN-planes of $H$ that are contained in $W$;

\item There exists a $p \in W$ which is contained in only finitely many GN-planes of $H$, but there does 
not exist a nonzero $c \in V(H)$ which shares a GN-plane of $H$ with every $q \in W$;

\item $\rk \jac H \le 1$ or $W$ is properly contained in the linear span of two GN-planes of $H$ 
which are contained in $W$;

\item $W$ is a properly contained in a $4$-dimensional linear subspace of $\C^n$ and $\rk \jac H \le 3$;

\end{enumerate}
we have {\upshape (1)} $\Rightarrow$ {\upshape (2)} $\Rightarrow$ {\upshape (3)} $\Rightarrow$ 
{\upshape (4)} $\Rightarrow$ {\upshape (5)}.
\end{corollary}

\begin{proof}
From lemma \ref{Wrk}, it follows that $W$ is irreducible and that $\rk \jac H = \dim W$. 
\begin{description}
 
\item[(1) \imp (2)]
Assume that $\dim V(H) = \rk \jac H \le 3$ and that (2) does not hold. Then there exists a nonzero
$p \in W$ which contains infinitely many GN-planes of $H$ that are contained in $W$. Suppose that
$W$ is the zero set of $g_1, g_2, \ldots, g_m$ and let
$$
Y = \{q \in W \mid g_1(p + tq) = g_2(p + tq) = \cdots = g_m(p + tq)= 0\}
$$
Then $Y$ has an irreducible component $Z$ which contains infinitely many GN-planes of $H$.
Hence $\dim Z \ge 3$. Since $Z \subseteq Y \subseteq W$ and $\dim W = \rk \jac H \le 3$, 
it follows from the irreducibility of $Z$ and $W$ that $Z = W$. So $p$ is a nonzero (projective) 
apex of $W$ and (1) does not hold. 

\item[(2) \imp (3)] 
Assume that $\dim V(H) = \rk \jac H$ and that (3) does not hold. Since $\dim V(H) = \rk \jac H$, 
it follows from lemma \ref{Wrk} that $W$ is an irreducible component of $V(H)$, 
so the interior of $W$ as a closed subset of $V(H)$ is nonempty. Take $p$ in that interior and let 
$L_p \ni p$ be a GN-plane of $H$. 
Since $L_p$ is irreducible, $L_p$ is contained in an irreducible component of $V(H)$, 
which can only be $W$ because $W$ is the only irreducible component of $V(H)$ which contains $p$.

So if $p$ is contained in infinitely many GN-planes of $H$, then (2) cannot hold. 
Hence assume that $p$ is contained in only finitely many GN-planes of $H$. Since (3) does not hold, 
there exists a nonzero $c \in V(H)$ which shares a GN-plane of $H$ with every $q \in W$.
Inductively, we can choose $p^{(i)}$ in the interior of $W$ outside $L_{p^{(1)}}, L_{p^{(2)}}, 
\ldots, L_{p^{(i-1)}}$ for $i = 1,2,3,\ldots$, such that $c \in L_{p^{(i)}}$ for each $i$.
As we have seen above, $L_{p^{(i)}} \subseteq W$ for each $i$, so $c$ is a counterexample
to the claim of (2).

\item[(3) \imp (4)]
Assume that (3) is satisfied.
From (ii) of theorem \ref{Lpth}, it follows that there exist a $p^{(1)} \in W$ and a $p^{(2)} \in W$ 
as in (iii) of theorem \ref{Lpth}. Take $L_{p^{(1)}}$ and $L_{p^{(2)}}$ as in (iii) of theorem \ref{Lpth}. 
Since there is no nonzero $c \in V(H)$ which shares a GN-plane of $H$ with every $q \in W$, 
$W$ cannot be equal to any linear span. Hence it suffices to show that $W$ is
contained in the linear span of $L_{p^{(1)}}$ and $L_{p^{(2)}}$. In the case where 
$L_{p^{(1)}} \cap L_{p^{(2)}} = \{0\}$, this follows directly from (iii) of theorem \ref{Lpth},
so assume that there exist a nonzero $c \in L_{p^{(1)}} \cap L_{p^{(2)}}$. Let
$$
Y = \{q \in W \mid H(c + tq) = 0\}
$$
From (3), it follows that $Y$ is a proper algebraic subset of $W$. Since $W$ irreducible and 
contained in the union of $Y$ and the linear span of $L_{p^{(1)}}$ and $L_{p^{(2)}}$, $W$
is contained in the linear span of $L_{p^{(1)}}$ and $L_{p^{(2)}}$.

\item[(4) \imp (5)]
Assume that (4) is satisfied. If $\rk \jac H \le 1$, then $W$ is a line through the 
origin on account of (i) of theorem \ref{hmgqt5th}, which gives (5). So assume that 
$\rk \jac H \ge 2$. Then $W$ is properly contained in a $4$-dimensional linear subspace of 
$\C^n$ and hence $\rk \jac H = \dim W < 4$. \qedhere

\end{description}
\end{proof}

\begin{remark}
Theorem \ref{hmgqt5th} was obtained in \cite[p.\@ 565]{gornoet}, but Gordan and N{\"o}ther
proved additionally that $\jac H \cdot H(y) = 0$ if $\rk \jac H \le 2$ and 
$\rk \jac H + \dim V(H) \le n$. See \cite[Th.\@ 4.1]{strongnil} for properties that are
equivalent to $\jac H \cdot H(y) = 0$.

The starting point of the distinction into cases `Fall a)' and `Fall b)' on 
\cite[p.\@ 565]{gornoet} is (i) of theorem \ref{Lpth}, but with the extra property that 
$L_p$ and $L_q$ are contained in $W$. Since $\dim V(H) = \rk \jac H = \dim W$
in this situation, this extra property can indeed be obtained, namely by extending
the genericity condition in the proof of (i) of theorem \ref{Lpth} by that $p$
and $q$ are in the interior of $W$ as a closed subset of $V(H)$.

The case where $k = 2$ of (iii) of theorem \ref{Lpth} is obtained on 
\cite[p.\@ 566]{gornoet}, and is used on the same page to prove the case where $n = 5$
and $\rk \jac H = 3$ of corollary \ref{Lpcor}. 
\end{remark}

\section{Homogeneous singular Hessians in dimension 5} \label{hess}

In \cite{gornoet}, Gordan and N{\"o}ther classified all homogeneous polynomials with singular Hessians
in dimension $5$ as follows.

\begin{theorem}[Gordan and N{\"o}ther] \label{gndim5}
Assume $h \in \C[x]$ is a homogeneous polynomial in dimension $n = 5$. 
If $\det \hess h = 0$ and $h$ is not a polynomial in $n - 1 = 4$ linear forms in $\C[x]$, 
then there exists an invertible matrix $T$ over $\C$ such that $h(Tx)$ is of the form
$$
h(Tx) = f\big(x_1,x_2,a_1(x_1,x_2)x_3 + a_2(x_1,x_2)x_4 + a_3(x_1,x_2)x_5\big)
$$
where $f$ and $a_1, a_2, a_3$ are polynomials over $\C$ in their arguments. 
\end{theorem}

The proof that is given below uses results about homogeneous quasi-trans\-lations in dimension 
five and follows the approach of Gordan and N{\"o}ther more or less.

The following connection exists between singular Hessians and quasi-trans\-lations.

\begin{proposition}[Gordan and N{\"o}ther] \label{hessprop}
Assume $h \in \C[x]$ such that $\det \hess h = 0$. Then there exists a nonzero $R \in \C[y]$
such that $R(\grad h) = 0$. For any such $R$, $x + H$ is a quasi-translation and 
$(\grad h)(x+tH) = \grad h$, where $H := (\grad R)(\grad h)$, and $H \ne 0$ if $R$ has minimum degree.
Furthermore, $h(x + tH) = h$ if $R^{*}(\grad h) = 0$ for every homogeneous part of $R^{*}$ of $R$.
\end{proposition}

\begin{proof}
From Proposition 1.2.9 of either \cite{arnobook} or \cite{homokema}, it follows that 
the components of $\grad h$ are algebraically dependent over $\C$, so $R$ indeed exists.
By the chain rule, 
$$
\jac H \cdot H = (\hess R)|_{y=\grad h} \cdot \hess h \cdot H
$$
So if $\hess h \cdot H = 0$, then $x + H$ is a quasi-translation on 
account of (3) $\Rightarrow$ (1) of proposition \ref{qtprop}. Indeed, if we take the Jacobian of
$R(\grad h) = 0$, we obtain
$$
\jac 0 = \jac \big(R(\grad h)\big) = (\jac R)_{y = \grad h} \cdot \hess h = H\tp \cdot \hess h
$$
which gives $\hess h \cdot H = 0$, because $\hess h$ is symmetric. Furthermore,
\eqref {fxtHeqv} in proposition \ref{qtprop} tells us that $(\grad h)(x + tH) = 0$. 

If $R$ has minimum degree and $H_i = 0$, then $\parder{}{y_i} R = 0$ because
$(\parder{}{y_i} R)(\grad h) = H_i = 0$. Since $R \notin \C$, we see that $H \ne 0$
if $R$ has minimum degree.

Suppose that $R^{*}(\grad h) = 0$ for every homogeneous part $R^{*}$ of $R$. Let 
$E_y: \C[y] \rightarrow \C[y]$ be the
map which multiplies each term by its own degree in $y$. Then one can verify that $E_y R = y\tp \grad R$,
and that $E_y R$ is a linear combination of the homogeneous parts $R^{*}$ of $R$.
So $\jac h \cdot H = (y\tp \grad R)_{y = \grad h} = (E_y R)_{y = \grad h} = 0$.
Hence $h(x+tH) = h$ on account of \eqref {fxtHeqv} in proposition \ref{qtprop}.
\end{proof}
 
In order to prove theorem \ref{gndim5}, we need the classification of all homogeneous polynomials 
with singular Hessians in dimensions less than $5$, which is as in theorem \ref{hessdim4hmg} below.
Our proof of theorem \ref{hessdim4hmg} is somewhat different from that by Gordan and N{\"o}ther.
A proof of theorem \ref{hessdim4hmg} which is based on that by Gordan and N{\"o}ther
can be found in \cite{gnlossen}. 

\begin{theorem}[Gordan and N{\"o}ther] \label{hessdim4hmg}
Assume $h \in \C[x]$ is a homogeneous polynomial in dimension $n \le 4$. 
If $\det \hess h = 0$, then the components of $\grad h$ are linearly
dependent over $\C$.
\end{theorem}

\begin{proof} 
Suppose that the components of $\grad h$ are linearly independent over $\C$.
Then $\deg \grad h \ge 1$ because $\det \hess h = 0$.
Let $H = (\grad R)(\grad h)$ as in proposition \ref{hessprop}, such that $R$ has minimum 
degree. Then $H$ is a nonzero quasi-translation and $\deg H \ge 1$ because $\deg R \ge 2$ and
$\deg \grad h \ge 1$. Furthermore, $H$ is homogeneous because 
$R$ and $\grad h$ are homogeneous.
From (2) $\Rightarrow$ (3) of proposition \ref{hmgprop}, it follows 
that $r := \rk \jac H \le \max\{n-2,1\} \le 2$. Using (i) and (iii) of theorem \ref{hmgqt5th}, 
we can deduce that $x + H$ has $n-r < n$ linear invariants. 

Since $n-r < n$, there exists a nonzero $p \in \C^n$ which is a zero of all these $n - r$ linear invariants.
From Proposition 1.2.9 of either \cite{arnobook} or \cite{homokema}, it follows that $\trdeg_{\C}(H) = r$.
Hence the $n - r$ linear invariants of $x + H$ generate the ideal $(\tilde{R} \in \C[y] \mid \tilde{R}(H) = 0)$ 
of $\C[y]$. Consequently, $p$ is a projective image apex of $H$. From lemma \ref{B} below, it follows that 
$\jac h \cdot p = 0$, so the components of $\grad h$ are linearly dependent over $\C$ indeed.
\end{proof}

\begin{lemma} \label{B}
Let $h \in \C[x]$ and $R \in \C[y]$, such that $R^{*}(\grad h) = 0$ for every homogeneous
part $R^{*}$ of $R$. 

Then $\jac h \cdot \jac H = 0$, where $H := (\grad R)(\grad h)$. Furthermore, if
$p$ is a projective image apex of $H$, then $\jac h \cdot p = 0$.
\end{lemma}

\begin{proof}
From proposition \ref{hessprop}, it follows that $h(x + tH) = h$. By taking the Jacobian on both
sides, we obtain 
$$
(\jac h)|_{x=x+tH} \cdot (I_n + t \jac H) = \jac h
$$
From proposition \ref{hessprop} again, it follows that $(\jac h)|_{x=x+tH} = \jac h$, so
$\jac h \cdot t \jac H = 0$, which gives the first claim. 

Suppose that $p$ is a projective image apex of $H$. Take $T \in \GL_n(\C)$ such that the last
column of $T$ equals $p$. Then $e_n$ is a projective image apex of $\tilde{H} := T^{-1} H$.
So $\tilde{H}_n$ is algebraically independent of $\tilde{H}_1, \tilde{H}_2, \ldots, \tilde{H}_{n-1}$.
Hence $\trdeg_{\C} \C(\tilde{H}) = \trdeg_{\C} \C(\tilde{H}_1,\tilde{H}_1,\ldots,\tilde{H}_{n-1}) + 1$. 
From proposition 1.2.9 of either \cite{arnobook} or \cite{homokema}, it follows that the last row 
of $\jac \tilde{H}$ is independent of the rows above it. 

But $\jac h \cdot T \cdot \jac \tilde{H} = \jac h \cdot \jac H = 0$. Hence the rightmost entry of 
$\jac h \cdot T$ is zero. So $\jac h \cdot p = 0$ indeed.
\end{proof}

Theorem \ref{gndim5} is formulated as \cite[Th.\@ 3.6]{singhess}. The starting point of 
the proof of \cite[Th.\@ 3.6]{singhess} is \cite[Th.\@ 2.1 iii)]{singhess}, which is
not accompanied by a proof and comes down the following.

\begin{theorem}[Gordan and N{\"o}ther] \label{gnqtdim5}
Assume $h \in \C[x]$ is a homogeneous polynomial in dimension $n = 5$. Suppose that
$R(\grad h) = 0$, such that $R$ has minimum degree. Then $R$ can be expressed as a polynomial in
three linear forms over $y$.
\end{theorem}

\begin{proof}
Notice that $R$ is homogeneous because $h$ is homogeneous and $R$ has minimum degree.
We distinguish two cases.
\begin{itemize}

\item \emph{$R$ cannot be expressed as a polynomial in four linear forms over $y$.} \\
Then the components of $\grad R$ are linearly independent over $\C$. Since $R$
has minimum degree, the components of $H := (\grad R)(\grad h)$ are linearly independent 
over $\C$ as well. Write $H = g \tilde{H}$, where $g \in \C[x]$, such that 
$\gcd\{\tilde{H}_1,\tilde{H}_2,\ldots,\allowbreak  \tilde{H}_n\} = 1$. Since
the components of $H$ and hence also $\tilde{H}$ are linearly independent over $\C$, 
we have $\deg \tilde{H} \ge 1$. On account of proposition \ref{irred}, 
$\rk \jac \tilde{H} = \rk \jac H$.

Since the components of $\tilde{H}$ are linearly independent over $\C$, it follows from 
theorem \ref{hmgqt5th} that $3 \le \rk \jac \tilde{H} \le \dim V(\tilde{H}) \le n-2$, so 
$\rk \jac \tilde{H} = \dim V(\tilde{H}) = 3$.
From (1) $\Rightarrow$ (5) of corollary \ref{Lpcor}, it follows that 
$\tilde{H}$ has a projective image apex, say $p$. Then $f(\tilde{H}) = 0$ implies
$f(\tilde{H} + tp) = 0$ for every homogeneous $f \in \C[y]$. Hence 
$f(H) = 0$ implies $f(H + tgp) = 0$ for every homogeneous $f \in \C[y]$. Since
$H$ is homogeneous, we can substitute $t = g^{-1}t$ to deduce that $p$ is 
a projective image apex of $H$ as well.

From lemma \ref{B}, it subsequently follows that
$\jac h \cdot p = 0$. Hence the components of $\grad h$ are linearly dependent over $\C$.
Since $R$ has minimum degree, we conclude that $\deg R = 1$, so $R$ is a linear form in $\C[y]$. 
Contradiction. 

\item \emph{$R$ can be expressed as a polynomial in four linear forms over $y$.} \\
Then there is an $i \le 5$ such that $y_i$ is not a linear combination of these four
linear forms. Say that $i = 5$. Then $R$ is of the form $\tilde{R}(y_1 + c_1 y_5,
y_2 + c_2 y_5, y_3 + c_3 y_5, y_4 + c_4 y_5)$, where $c_i \in \C$ for each $i$.
Furthermore, $\tilde{R} \in \C[y_1,y_2,y_3,y_4]$ is homogeneous and 
$\tilde{R}(\grad \tilde{h}) = \tilde{R}(\grad \hat{h}) = 0$, where
$$
\tilde{h} = h\big|_{x_5 = x_5 + c_1 x_1 + c_2 x_2 + c_3 x_3 + c_4 x_4} \qquad \mbox{and} \qquad 
\hat{h} = \tilde{h}\big|_{x_5 = 1}
$$
Since $\tilde{h}$ is homogeneous, say of degree $d$,
it follows that $\tilde{h} = x_5^d \hat{h}(x_5^{-1} x)$ and that $\grad \tilde{h}$ and 
$x_5^{d-1} (\grad \hat{h})(x_5^{-1} x)$ agree on the first $4$ components.
From this, we can deduce that $\tilde{R}$, as a homogeneous
polynomial in $\C[y_1, y_2, y_3, y_4]$ such that $\tilde{R}(\grad \hat{h}) = 0$,
has minimum degree as well.
From theorem \ref{A} below, we obtain that $\tilde{R}$ can be expressed as a polynomial in 
three linear forms in $\C[y_1, y_2, y_3, y_4]$. Hence $R$ can be expressed as a polynomial
in three linear forms in $\C[y]$. \qedhere

\end{itemize}
\end{proof}

\begin{theorem} \label{A}
Let $n = 4$ and $h \in \C[x]$, not necessarily homogeneous. 
Suppose that $R \in \C[y]$ is homogeneous, such that $R(\grad h) = 0$. 
If $R$ has minimum degree, then $R$ can be expressed as a polynomial in three linear forms
in $\C[y]$.
\end{theorem}

\begin{proof}
Suppose that $R$ has minimum degree.
Let $\bar{h}$ be the leading homogeneous part of $h$, and define $H := (\grad R)(\grad h)$.
From proposition \ref{hessprop}, it follows that $h(x + tH) = h$. By taking the leading coefficient 
with respect to $t$, we deduce that $\bar{h}(H) = 0$. 

Since $\bar{h}$ is homogeneous and $R(\grad \bar{h}) = 0$, it follows from
theorem \ref{hessdim4hmg} that the components of 
$\grad \bar{h}$ are linearly dependent over $\C$, say that $L(\grad \bar{h}) = 0$ 
for some linear form $L \in \C[y]$. Assume first that $\rk \hess \bar{h} = 3$.
Then the relations between the components of $\grad \bar{h}$ form a prime ideal
of height one, which is a principal ideal because $\C[y]$ is a unique factorization
domain. Since $L$ is irreducible, $(L)$ must be that principal ideal, and $L \mid R$ because 
$R(\grad \bar{h}) = 0$. Since $R$ has minimum degree, $R$ is irreducible, so $R$ is linear. 

Assume next that $\rk \hess \bar{h} \le 2$. Since there exists a linear relation
between the components of $\grad \bar{h}$, there exists a $T \in \GL_n(\C)$ such that
the last component of $T\tp \grad \bar{h}$ is zero. Hence the last component of 
$\grad (\bar{h}(Tx)) = T\tp (\grad \bar{h})(Tx)$ is zero. So $\bar{h}(Tx) \in \C[x_1,x_2,x_3]$. Since
$\hess (\bar{h}(Tx)) = T\tp (\hess \bar{h})|_{x = Tx} T$, we see that $\rk \hess (\bar{h}(Tx)) \le 2$.
It follows from theorem \ref{hessdim4hmg} again that $\bar{h}(Tx)$ can be expressed as a polynomial in 
two linear forms. Hence $\bar{h} = \bar{h}\big(T(T^{-1}x)\big)$ can be expressed as a polynomial in
two linear forms as well.

Since $\bar{h}$ is homogeneous in addition, $\bar{h}$ decomposes into linear factors, and
one of these factors is already a relation between $H_1, H_2, H_3, H_4$.
So there exist a linear form $M \in \C[x]$ such that $M\big((\grad R)(\grad h)\big) = M(H) = 0$.
Since $R$ has minimum degree, $M(\grad R) = 0$. On account of Example 1.2 in \cite{singhess},
$R$ can be expressed as a polynomial in three linear forms over $y$.
\end{proof}

\begin{remark} \label{rem}
The proof of the first case in the proof of theorem \ref{gnqtdim5} is different from that 
given in \cite[p.\@ 568]{gornoet}, where the second claim of lemma \ref{B} is obtained by way of
differentiation on the inverse of $H$. 
Since the inverse of $H$ is not a map, the above proof of this first case seems much easier.
The proof of this first case as given in \cite[Th.\@ 5.3.7]{homokema} is incorrect.

The proof of the second case in the proof of theorem \ref{gnqtdim5} comes from 
\cite[p.\@ 567]{gornoet}. This seems a little odd, because lemma \ref{A} is about
not necessarily homogeneous polynomials, which Gordan and N{\"o}ther did not consider in
\cite{gornoet}. But in spite of that, the proof of lemma \ref{A} comes from \cite[p.\@ 567]{gornoet} 
indeed. 

On \cite[p.\@ 567]{gornoet}, Gordan and N{\"o}ther additionally prove that $\rk \jac H \le 2$,
as follows. They assume that $H_1 = H_2 = 0$ on account of theorem \ref{gnqtdim5} and proposition 
\ref{qtconj}, and show the first claim of lemma \ref{B} that $\jac h \cdot \jac H = 0$, to conclude 
that either $h \in \C[x_1,x_2]$ or that the rows of $\jac (H_3,H_4,H_5)$ are dependent. In both 
cases, $\rk \jac H \le 2$ indeed, because $H_i \in \C[\parder{}{x_1} h, \parder{}{x_2} h]$ for 
all $i$ in the first case, so that the row space of $\jac H$ is generated by
$\jac (\parder{}{x_1} h)$ and $\jac (\parder{}{x_2} h)$.

Unlike Gordan and N{\"o}ther, we do not need to show that $\rk \jac H \le 2$ here, because for the 
techniques in \cite{singhess}, linear dependences between the components of $H$ are the only 
thing that matters. But the result of Gordan and N{\"o}ther can be used to fix the gap in
\cite{watanabe}, which is caused by the incorrect \cite[Lm.\@ 5.2]{watanabe},
and a gap on the same point in \cite{franch}. 
\end{remark}

\section{Homogeneous 5-dimensional quasi-translations of `Fall b)'} \label{Fallb}

In this section, we study homogeneous quasi-translations in dimension $5$ which
corresponds to `Fall b)' in \cite[\S 8]{gornoet}. 
In corollary \ref{Fallbcor}, we will show that homogeneous quasi-translations 
in dimension $5$ which are not of this type always have a linear invariant.

\begin{theorem} \label{Fallbth}
Assume $x + H$ is a homogeneous quasi-translation in dimension $5$ over $\C$, 
such that $\gcd\{H_1,H_2,H_3,H_4,H_5\} = 1$ and $H_5$ is algebraically independent
over $\C$ of $H_1, H_2, H_3, H_4$.

If $x + H$ does not have two independent linear invariants, then $\rk\jac H = \dim V(H) = 3$
and the following holds.
\begin{enumerate}[\upshape (i)]

\item $H$ is of the form
$$
H = \big(g h_1(p,q),g h_2(p,q),g h_3(p,q),g h_4(p,q),H_5\big)
$$
where $g \in \C[x]$, $h \in \C[y_1,y_2]^4$ and $(p,q) \in \C[x]^2$ are homogeneous, 
and $\gcd\{p,q\} = 1$.

\item $g, p, q$ are invariants of $x + H$, and $g \in \C[x_1,x_2,x_3,x_4]$.

\item $\deg_{x_5} H_5 \le \deg_{x_5} (H_1,H_2,H_3,H_4)$ and there exists a linear 
combination over $\C$ of $p$ and $q$ whose degree with respect to $x_5$ is less than 
$\max\{\deg_{x_5} p, \allowbreak \deg_{x_5} q\}$.

\item There exists an invariant $a \in \C[x]$ of degree at most $1$ of $x + H$, such that 
every invariant of $x + H$ which can be expressed as a polynomial in four linear forms in 
$\C[x]$ is contained in $\C[a]$. 

\item If $H$ has no linear invariants at all, then $g \in \C$.

\end{enumerate}
\end{theorem}

\begin{proof}
From (ii) of theorem \ref{hmgqt5th}, it follows that $2 \le \rk \jac H \le \dim V(H) \le 3$.
Assume that $x + H$ does not have two independent linear invariants.
From (iii) of theorem \ref{hmgqt5th}, it follows that $\rk \jac H \ne 2$, 
so $\rk\jac H = \dim V(H) = 3$. 
\begin{enumerate}[(i)]

\item Since $H_5$ is algebraically independent over $\C$ of 
$H_1, H_2, H_3, H_4$, it follows from $\rk \jac H = 3$ and Proposition 1.2.9 of either 
\cite{arnobook} or \cite{homokema} that $\rk \jac (H_1,H_2,\allowbreak H_3,H_4) = 2$. 
Using \cite[Th.\@ 2.2]{dp3} (see also \cite[Th.\@ 4.3.1]{homokema}), we see that $H$ is of the 
given form.

\item Take $i \le 4$ such that $H_i \ne 0$. Then $g \cdot h_i(p,q) = H_i$ and 
on account of proposition \ref{qtprop},
$$
\deg_t g(x + tH) + \deg_t h_i\big(p(x+tH),q(x+tH)\big) = \deg_t H_i(x + tH) = 0
$$
whence $\deg_t g(x + tH) = 0$ and $g$ is an invariant of $x + H$. 
Similarly, any linear form in $p$ and $q$ that divides $H_i$ is an invariant
of $x + H$ as well. If there is at most one independent linear form in
$p$ and $q$ that divides $H_i$ for any $i \le 4$ such that $H_i \ne 0$, then
$\deg g = \deg H$ and $x + H$ has three independent linear invariants,
which is a contradiction. Hence there are two independent linear forms in 
$p$ and $q$ that are invariants of $x + H$. Since $p$ and $q$ are in turn
linear forms in these invariants, $p$ and $q$ are invariants of $x + H$
themselves.

Since $g$ is an invariant of $x + H$, it follows from \eqref{fxtHeqv} in proposition 
\ref{qtprop} that $\jac g \cdot H = 0$. Hence
$$
g ~\Big|~ \jac g \cdot H - g \sum_{i=1}^4 h_i(p,q) \parder{}{x_i} g
= H_5 \parder{}{x_5} g 
$$
Now $\parder{}{x_5} g \ne 0$ contradicts the assumption that 
$\gcd\{g,H_5\} \mid \gcd\{H_1,\allowbreak H_2,H_3,H_4,H_5\} = 1$. Thus $g \in 
\C[x_1,x_2,x_3,x_4]$.

\item Let $r$ be the degree with respect to $x_5$ of $(H_1,H_2,H_3,H_4)$.
If the degree with respect to $x_5$ of $H_5$ is larger than $r+1$, then
$\deg_{x_5} \jac H_5 \cdot H = \deg_{x_5} (\parder{}{x_5} H_5) \cdot H_5 > 2r + 1$, 
which contradicts (3) $\Rightarrow$ (1) of proposition \ref{qtprop}. Take for 
$\bar{H}_i$ all terms of degree $r$ with respect to $x_5$ of $H_i$ if $i \le 4$, and
for $\bar{H}_5$ all terms of degree $r+1$ with respect to $x_5$ of $H_5$.

Then the part of degree $2r$ with respect to $x_5$ of 
$\jac (H_1,H_2,H_3,H_4) \cdot H$ equals $\jac (\bar{H}_1,\bar{H}_2,
\bar{H}_3,\bar{H}_4) \cdot \bar{H}$ and the part of degree $2r+1$ of
$\jac H_5 \cdot H$ equals $\jac \bar{H}_5 \cdot \bar{H}$. Since $\jac H_i \cdot H = 0$ 
for all $i$ on account of (1) $\Rightarrow$ (3) of proposition \ref{qtprop}, we have
$\jac \bar{H} \cdot \bar{H} = 0$.

On account of (3) $\Rightarrow$ (2) of proposition \ref{qtprop}, 
$\deg_t \bar{H}_5(x + t\bar{H}) = 0$. Since $x_5 \mid \bar{H}_5$,
$\deg_t (x + t\bar{H})_5 = 0$ as well. Hence $\bar{H}_5 = 0$ and 
$\deg_{x_5} H_5 \le r = \deg_{x_5} (H_1,\allowbreak H_2,H_3,H_4)$. 
By taking leading parts with respect to $x_5$, we see that for homogeneous 
and hence any $R \in \C[y_1,y_2,y_3,y_4]$, $R(H_1,H_2,H_3,\allowbreak H_4) = 0$ implies 
$R(\bar{H}_1,\bar{H}_2,\bar{H}_3,\bar{H}_4) = 0$. It follows from Proposition 1.2.9 of 
either \cite{arnobook} or \cite{homokema} that 
\begin{align*}
\rk \jac (x_5^{-r}\bar{H}) 
&= \trdeg_{\C} \C(x_5^{-r}\bar{H}_1,x_5^{-r}\bar{H}_2,x_5^{-r}\bar{H}_3,x_5^{-r}\bar{H}_4) \\
&\le \trdeg_{\C} \C(H_1,H_2,H_3,H_4) = \rk \jac (H_1, H_2, H_3, H_4) = 2
\end{align*}
Since $x_5^{-r} \bar{H}_5 = 0$ and $x_5^{-r} \bar{H}_i \in \C[x_1,x_2,x_3,x_4]$,
we deduce that $x + \bar{H}$ can be regarded as a quasi-translation in dimension four 
(over its first four coordinates). By (i) and (iii) of theorem \ref{hmgqt5th}, 
there are two independent linear forms $l_1$ and $l_2$ in 
$x_1,x_2,x_3,x_4$ such that $l_1(x_5^{-r}\bar{H}) = l_2(x_5^{-r}\bar{H}) = 0$.
So $l_1(\bar{H}) = l_2\bar{H}) = 0$.

Suppose that the leading parts of $p$ and $q$ with respect to $x_5$ are independent
and of the same degree with respect to $x_5$. Since $(\bar{H}_1,\bar{H}_2,\bar{H}_3,
\bar{H}_4)$ is the leading part of $(H_1, H_2, H_3, H_4)$ with respect to $x_5$, it 
follows that $(\bar{H}_1,\bar{H}_2,\bar{H}_3,\bar{H}_4) = h(\bar{p},\bar{q})$, where $\bar{p}$ and 
$\bar{q}$ are the leading parts of $p$ and $q$ respectively with respect to $x_5$. By assumption,
$\bar{p}$ and $\bar{q}$ are independent, so we can deduce from $l_1(\bar{H}) = l_2(\bar{H}) = 0$
that $l_1(h) = l_2(h) = 0$ and hence also $l_1(H) = l_2(H) = 0$. Contradiction,
thus the leading parts of $p$ and $q$ with respect to $x_5$ are dependent
or have different degrees with respect to $x_5$, as desired.

\item Take for $a$ the linear invariant of $x + H$, if it has any, and take $a = 1$ otherwise.
Let $f$ be a non-constant invariant of $x + H$ which can be expressed in four linear forms. 
We distinguish two cases.
\begin{itemize}
 
\item $f \in \C[x_1,x_2,x_3,x_4]$. \\ 
On account of (iii) above, we can obtain that $\deg_{x_5} p < \deg_{x_5} q$, 
namely by replacing $p$ and $q$ by linear combinations of $p$ and $q$, and
adapting $h$ accordingly.
If we replace $H$ by $T^{-1}H(Tx)$ and $(f,p,q)$ by $(f(Tx),p(Tx),q(Tx))$ for some 
$T \in \GL_5(\C)$ such that the last column of $T$ is equal to the fifth unit 
vector, the form of $H$ does not change and neither do $\deg_{x_5} f, \deg_{x_5} p$ and 
$\deg_{x_5} q$. By choosing $T$ appropriate, we can obtain 
$-\infty \le \deg_{y_2} h_1 < \deg_{y_2} h_2 < \deg_{y_2} h_3 < \deg_{y_2} h_4$.

On account of \eqref{fxtHeqv} in proposition \ref{qtprop}, $\jac f \cdot H = 0$. 
By looking at the leading coefficient with respect to $x_5$ in
$\jac f \cdot H$, we can successively deduce that $\parder{}{x_4} f = 0$,
$\parder{}{x_3} f = 0$, $\parder{}{x_2} f = 0$, and $H_1 = 0$. Hence $f$
is a polynomial over $\C$ in the invariant $x_1$ of $x + H$, and $f$ was
a polynomial over $\C$ in the invariant $(T^{-1})_1 x$ of $x + H$ before 
replacing $H$ by $T^{-1}H(Tx)$. Since $x + H$ does not have two independent
linear invariants, we see that $f \in \C[a]$.

\item $f \notin \C[x_1,x_2,x_3,x_4]$. \\
There exists a $T \in \GL_5(\C)$ such that $f(Tx) \in \C[x_1,x_2,x_3,x_5]$ and
last column of $T$ is equal to the fifth unit vector. Just as above, we replace $H$ 
by $T^{-1}H(Tx)$ and $(f,p,q)$ by $(f(Tx),p(Tx),q(Tx))$. So we may assume that 
$f \in \C[x_1,x_2,x_3,x_5]$. From \eqref{fxtHeqv}, it follows that $\jac f \cdot H = 0$ 
and that any homogeneous part of $f$ is an invariant of $x + H$ as well, so we 
may assume that $f$ is homogeneous.

Since $x + H$ has at most one linear invariant, we can use techniques in the 
proof of (i) of theorem \ref{hmgqt5th} to show that $\rk \jac (H_1,H_2,H_3) = 2$.
Hence the ideal $\bi := (R \in \C[y_1,y_2,y_3] \mid R(H_1,H_2,H_3) = 0)$ has height $1$, 
and since $\C[y]$ is a unique factorization domain, $\bi$ is principal. Say that
$R$ is a generator of $\bi$.

By looking at the leading homogeneous part of $f(x+H) = f$, we see that $f(H) = 0$.
Since $H_5$ is algebraically independent of $H_1,H_2,H_3$, we deduce that 
$R(x_1,x_2,x_3) \mid f$. From \eqref{fxtHeqv}, it follows $f(x + tH) = f$, from
which we can deduce that every factor of $f$ is an invariant of $x + H$. The case
$f \in \C[x_1,x_2,x_3,x_4]$ above tells us that $R(x_1,x_2,x_3) \in \C[a]$, and 
$f / R(x_1,x_2,x_3) \in \C[a]$ follows by induction on the degree of $f$. 

\end{itemize}
\item From (ii), it follows that $g \in \C[x_1,x_2,x_3,x_4]$. On account of (iv),
$g \in \C[a]$, where $a$ is as in (iv). If $H$ has no linear 
invariant, then $\deg a = 0$. Hence $g \in \C[a] = \C$ if $H$ has no linear invariant. 
\qedhere

\end{enumerate}
\end{proof}

\begin{corollary} \label{Fallbcor}
Assume $x + H$ is a homogeneous quasi-translation over $\C$ in dimension $5$ without
linear invariants. Then $\deg H \ge 12$. More precisely, there exists a
$T \in \GL_5(\C)$ such that $T^{-1} H(Tx)$ is of the form
$$
T^{-1} H(Tx) = g \cdot \big(h_1(p,q),h_2(p,q),h_3(p,q),h_4(p,q),f\big)
$$
where $h$ is homogeneous of degree at least $3$ and $(p,q)$ is homogeneous of degree at least $4$. 
\end{corollary}

\begin{proof}
On account of proposition \ref{irred}, we may assume that 
$\gcd\{H_1,H_2,H_3,H_4,\allowbreak H_5\} = 1$. From (ii) of theorem \ref{hmgqt5th},
it follows that $\dim V(H) \le 3$. From (i) and (iii) of theorem \ref{hmgqt5th}, it follows that
$\rk \jac H \ge 3$. Using (2) $\Rightarrow$ (3) of proposition \ref{hmgprop}, we can deduce
that $\dim V(H) = \rk \jac H = 3$. From (1) $\Rightarrow$ (5) of corollary \ref{Lpcor}, we obtain that
$H$ has a nonzero (projective) image apex. From proposition \ref{qtconj}, it follows that we may 
assume that $e_5$ is a (projective) image apex.

From (i), (ii) and (v) of theorem \ref{Fallbth}, it follows that there are invariants $p$ and $q$ of 
$x + H$, such that $H$ is of the form 
\begin{equation} \label{Hpiaform}
H = \big(h_1(p,q),h_2(p,q),h_3(p,q),h_4(p,q),H_5\big) 
\end{equation}
such that $h$ and $(p,q)$ are homogeneous. Furthermore, it follows from (iii) and (iv) 
of theorem \ref{Fallbth} that we may assume that $\deg_{x_5} q > \deg_{x_5} p$
and $\deg_{x_5} p > 0$ respectively.

On account of \eqref{fxtHeqv} in proposition \ref{qtprop}, 
$q(x+tH) = q(x)$, and looking at the leading coefficient with respect to $t$ gives $q(H) = 0$. 
Since $e_5$ is a projective apex of $H$, we even have $q(H_1,H_2,H_3,H_4,H_5+t) = 0$. 
Hence $\deg_{x_5} q \le \deg q-1$ and in case of equality, looking at the leading 
coefficient with respect to $t$ in $q(x_1,x_2,x_3,x_4,t)$ gives a 
linear form $l_1$ such that $l_1(H_1, H_2, H_3, H_4)$, which contradicts that
$x + H$ has no linear invariants. Thus $\deg_{x_5} q \le \deg q-2$. If we combine this with
the conclusion of the previous paragraph, then we obtain
\begin{equation} \label{pqdeg}
0 < \deg_{x_5} p < \deg_{x_5} q \le \deg q - 2
\end{equation}
So $\deg (p,q) \ge \deg q \ge \deg_{x_5} q + 2 \ge 4$.

If $\deg h < 3$, then there exists a linear form
$l_2 \in \C[x_1,x_2,x_3,x_4]$ such that $l_2(h) = 0$ and hence also 
$l_2(H_1,H_2,H_3,H_4) = 0$, which contradicts that
$x + H$ has no linear invariants. Hence $\deg h \ge 3$. 
\end{proof}

The following theorem has been proved in \cite{liu} as well. The proof
that is given below is somewhat less computational than that in \cite{liu}.

\begin{theorem} 
Assume $x + H$ is a homogeneous quasi-translation over $\C$ in dimension $5$ without
linear invariants. Then $\deg H \ge 15$. More precisely, $\deg (p,q) \ge 5$, where
$p$ and $q$ are as in corollary {\upshape\ref{Fallbcor}}.
\end{theorem}

\begin{proof}
Just like in the proof of corollary \ref{Fallbcor}, we may assume that
$H$ is as in \eqref{Hpiaform} such that $h$ is homogeneous of degree at least $3$ and 
$(p,q)$ is homogeneous such that \eqref{pqdeg} is satisfied. 
If $\deg q \ge 5$, then $\deg H \ge 5 \deg h \ge 15$ indeed, because $(p,q)$ is homogeneous.
Hence assume that $\deg q \le 4$. We shall derive a contradiction.
\begin{enumerate}[(i)]

\item From (iv) of theorem \ref{Fallbth}, it follows that $\deg_{x_4} p \ge 1$ and 
$\deg_{x_4} q \ge 1$.
From \eqref{pqdeg}, we deduce that $\deg_{x_5} p = 1$, $\deg_{x_5} q = 2$ and $\deg q = 4$.
Assume without loss of generality that $\deg_{y_2} h_4 > \deg_{y_2} h_3 > \deg_{y_2} h_2 
> \deg_{y_2} h_1$. Then $\deg_{x_5} H_4 > \deg_{x_5} H_3 > \deg_{x_5} H_2 > \deg_{x_5} H_1$.

Let $r$ be the leading coefficient with respect to $x_5$ of $q$.
On account of \eqref{fxtHeqv} in proposition \ref{qtprop}, $\jac q \cdot H = 0$.
By looking at the leading coefficient with respect to $x_5$ of $\jac q \cdot
H = 0$, we deduce from (iii) of theorem \ref{Fallbth} that $r \in \C[x_1,x_2,x_3]$. 
Since $q(H_1,H_2,H_3,H_4,t) = 0$,
$r(H_1,H_2,H_3) = 0$ as well. By looking at the leading coefficient
with respect to $x_5$ in $r(H_1,H_2,H_3)$, we see that the coefficients of
$x_3^2$ and $x_2x_3$ of $r$ are zero. Hence $r$ is of the form
$r = (\lambda_1 x_1 + \lambda_2 x_2 + \lambda_3 x_3) x_1 - \lambda_4^2 x_2^2$, where
$\lambda_i \in \C$ for all $i$.
Since $r$ is irreducible, we have $\lambda_3 \lambda_4 \ne 0$. 

\item We show that for invariants $f$ of $x + H$, we have $\deg_{x_4,x_5} f = 
\deg_{x_5} f$. Let $f$ be an invariant of $x + H$. From \eqref{fxtHeqv} in
proposition \ref{qtprop}, it follows that $f(x + tH) = 0$.
Let $\bar{f}$ be the leading part of $f$ with respect to
$(x_4,x_5)$ and suppose that $\deg_{x_4,x_5} f > \deg_{x_5} f$. Then $x_4 \mid \bar{f}$, say
that $x_4^v \mid \bar{f}$ and $x_4^{v+1} \nmid \bar{f}$. On account of 
(iii) of theorem \ref{Fallbth}, $\deg_{x_5} H_4 \ge \deg_{x_5} H_5$ and $\deg_{x_5} H_4 > 
\deg_{x_5} H_i$ for all $i \le 3$. So $\deg_{x_5} H_4 \ge \deg_{x_5} H_i -
\deg_{x_5} x_i - 1$ for all $i \ne 4$.

If we change a factor $x_i$ in a product into $t H_i$, the degree with respect to $x_5$
of that product will increase $\deg_{x_5} tH_4 - \deg_{x_5} x_4 = \deg_{x_5} tH_4$ if $i = 4$ 
and $\deg_{x_5} t H_i - \deg_{x_5} x_i \le \deg_{x_5} tH_4 - 1$ if $i \ne 4$.
Having to do such a change $v$ times, starting with a term $u \in \C[x]$, we deduce from 
$\deg_{x_5} u = \deg_{x_4,x_5} u - \deg_{x_4} u$ for terms $u \in \C[x]$ that for any term 
and hence any polynomial $u \in \C[x]$, the coefficient of $t^v$ of $u(x + tH)$ has degree at most 
\begin{align*}
b(u) &:=  \deg_{x_4,x_5} u - \deg_{x_4} u  + \deg_{x_4} u \cdot \deg_{x_5} tH_4 + {}\\
&\qquad (v - \deg_{x_4} u) \cdot \big(\deg_{x_5} tH_4 - 1\big) + 
\\
&\hphantom{:}= v\cdot\big(\deg_{x_5} tH_4 - 1\big) + \deg_{x_4,x_5} u
\end{align*}
with respect to $x_5$. Since $b(u)$ is affinely linear in $\deg_{x_4,x_5} u$ as a function on
terms $u \in \C[x]$, the part of degree $b(f)$ with respect to $x_5$ of the coefficient of $t^v$ of 
$f(x + tH)$ is equal to that of $\bar{f}(x + tH)$. 

The part of degree $v$ with respect to $t$ of $\bar{f}(x_1,x_2,x_3,x_4+t H_4,x_5)$ equals
$(tH_4)^v \frac{1}{v!} \parder[v]{}{x_4} \bar{f}$. By definition of $v$, 
\begin{align*}
\deg_{x_5} \Big((tH_4)^v \frac{1}{v!} \parder[v]{}{x_4} \bar{f}\Big)
&= v\cdot \deg_{x_5} tH_4 + \deg_{x_5} \parder[v]{}{x_4} \bar{f} \\
&= v\cdot \deg_{x_5} tH_4 - v + \deg_{x_4,x_5} \bar{f} \\ 
&= b(\bar{f}) = b(f)
\end{align*}
so the part of degree $b(f)$ with respect to $x_5$ of the coefficient of $t^v$ of 
$\bar{f}(x_1,x_2,x_3,x_4+t H_4,x_5)$ is nonzero. Furthermore, we can deduce from
$\deg_{x_5} tH_4 - \deg_{x_5} x_4 = \deg_{x_5} tH_4 > \deg_{x_5} t H_i - \deg_{x_5} x_i$ for all $i \neq 4$
and $\deg_{x_5} \parder[v]{}{x_4} \bar{f} = \deg_{x_5} \bar{f}$ that the degree with respect to 
$x_5$ of $\bar{f}(x + tH) - \bar{f}(x_1,x_2,x_3,x_4+t H_4,x_5)$ is less than $b(f)$. Hence
$$
\deg_{x_5} f(x + tH) = \deg_{x_5} \bar{f}(x + tH) = \deg_{x_5} \bar{f}(x_1,x_2,x_3,x_4+t H_4,x_5) = b(f)
$$
But the coefficient of $t^v$ of $f(x + tH)$ is zero, so $v = 0$. Hence $\deg_{x_4,x_5} f = 
\deg_{x_5} f$ for invariants $f$ of $x + H$.

Since $p$ and $q$ are invariants of $x + H$, and $\deg_{x_5} p = 1$ and
$\deg_{x_5} q = 2$, we have $\deg_{x_4,x_5} p = 1$ and $\deg_{x_4,x_5} q = 2$.

\item Let $\bar{H}_i$ be the part of $H_i$ that has degree $2\deg h - 1$ with respect 
to $(x_4,x_5)$, for $i = 1,2,3$, and $\bar{H}_j$ the part of $H_j$ that has degree 
$2 \deg h$ with respect to $(x_4,x_5)$, for $j = 4,5$. Then the part of degree
$4\deg h - 2$ with respect to $(x_4,x_5)$ of 
$\jac (H_1,H_2,H_3,H_4) \cdot H$ equals $\jac (\bar{H}_1,\bar{H}_2,
\bar{H}_3,\bar{H}_4) \cdot \bar{H}$ and the part of degree $4\deg h - 1$ with respect to 
$(x_4,x_5)$ of $\jac H_5 \cdot H$ equals $\jac H_5 \cdot H$. Since $\jac H \cdot H = 0$ on 
account of (1) $\Rightarrow$ (3) of proposition \ref{qtprop}, we have
$\jac \bar{H} \cdot \bar{H} = 0$.

On account of (iii), $\deg_{x_5} H_i = \deg_{x_4,x_5} H_i$ for all $i$.
Consequently, $\bar{H}_1 = \bar{H}_2 = 0$ and $\bar{H}_3$ and $(\bar{H}_4,\bar{H}_5)$ 
are homogeneous with respect to $(x_4,x_5)$. We shall show that $\bar{H}_5$
is linearly dependent over $\C$ of $\bar{H}_4$, distinguishing the cases
$\bar{H}_3 = 0$ and $\bar{H}_3 \ne 0$.

Assume first that $\bar{H}_3 = 0$. From (\ref{qtnilp}), it follows that
$\jac \bar{H}$ is nilpotent. Since $\bar{H}_1 = \bar{H}_2 = \bar{H}_3 = 0$, 
$\jac_{x_4,x_5} (\bar{H}_4, \bar{H}_5)$ is nilpotent as well. From 
\cite[Th.\@ 7.2.25]{arnobook} (see also \cite{esshubnc}), we obtain that
$$
(\bar{H}_4, \bar{H}_5) = (bc (a x_4 - b x_5)^{2 \deg h}, ac (a x_4 - b x_5)^{2 \deg h})
$$
Since $\deg (\bar{H}_4, \bar{H}_5) = 2 \cdot 2 \deg h$, this is only possible
if $a$ and $b$ are constant. Hence $ \bar{H}_5 = b^{-1} a \bar{H}_4$ is 
linearly dependent over $\C$ of $\bar{H}_4$.

Assume next that $\bar{H}_3 \ne 0$. Let $\bar{q}$ be the leading and quadratic part of 
$q$ with respect to $(x_4,x_5)$. Then $\bar{q} \mid \bar{H}_4$, so 
$\deg_t \bar{q}(x + t\bar{H}) \le \deg_t \bar{H}_4(x + t\bar{H}) = 0$ on account of
(3) $\Rightarrow$ (2) of proposition \ref{qtprop}. Since $\lambda_3 \ne 0$ and 
the leading term with respect to $x_5$ of $q$ is divisible by $r$, 
we have $\deg_{x_3,x_4,x_5} \bar{q} = \deg_{x_4,x_5} \bar{q} + 1 = 3$.
The coefficient of $t^3$ in $\bar{q}(x+t\bar{H})$ is of the form $x_1\bar{H}_3
s(\bar{H}_4,\bar{H}_5)$, where $s$ is a quadratic form, which decomposes into linear factors.
Since $\bar{H}_3 \ne 0$, we deduce that $s(\bar{H}_4,\bar{H}_5) = 0$ and that 
$\bar{H}_5$ is linearly dependent over $\C$ of $\bar{H}_4$.

\item By way of a linear conjugation of $H$ and the same linear conjugation of 
$\bar{H}$, we can obtain $\bar{H}_5 = 0$. If $\bar{H}_3 = 0$, then one can compute that
$\deg_{x_5} \jac p \cdot H = \deg_{x_5} \parder{}{x_4} p \cdot \bar{H}_4$, which
gives a contradiction to \eqref{fxtHeqv} in proposition \ref{qtprop}. Hence $\bar{H}_3 \ne 0$. 
From \eqref{qtnilp}, it follows that $\jac \bar{H}$ is nilpotent. Since $\bar{H}_1 = \bar{H}_2 = 
\bar{H}_5 = 0$, $\jac_{x_3,x_4} (\bar{H}_3, \bar{H}_4)$ is nilpotent as well. From 
\cite[Th.\@ 7.2.25]{arnobook} (see also \cite{esshubnc}), we obtain that
$$
(\bar{H}_3, \bar{H}_4) = \big(b g(a x_3 - b x_4) + c, ag (a x_3 - b x_4) + d\big)
$$
for certain $a,b,c \in \C[x_1,x_2,x_5]$. Hence $\deg_{x_4} \bar{H}_3 = \deg_{x_4} \bar{H}_4$. 
Since $\deg_{x_4} p \ge 1$ and $\deg_{x_4,x_5} p = 1$, this is only possible if $x_5 \mid a$ and
$\deg_{x_4} p = \deg_{x_4} \bar{q} = 1$. Since $\bar{q}$ be the leading and quadratic part of 
$q$ with respect to $(x_4,x_5)$, we deduce from $\deg_{x_4} \bar{q} = 1$ that $q$
has a term which is divisible by $x_4 x_5$, but no term which is divisible by $x_4^2$.
So $\deg_{x_4} p = \deg_{x_4} q = 1$ and the right hand side of
\begin{equation} \label{x4x5pq}
\deg_{x_5} \Big( \parder{}{x_4} p \Big) = 0 \qquad \mbox{and} 
\qquad \deg_{x_5} \Big( \parder{}{x_4} q \Big) = 1 
\end{equation}
follows. The left hand side of \eqref{x4x5pq} follows from $\deg_{x_4,x_5} p = 1$.

\item Since $q$ is an invariant of $x + H$, we obtain from proposition \ref{qtprop}
that $q(x+tH)\cdot H(x+tH) = q \cdot H$, and substituting $t = tq$ gives by way of 
(2) $\Rightarrow$ (1) of proposition \ref{qtprop} that
$x + qH$ is a quasi-translation. Since the leading coefficient with respect
to $x_4$ of $q$ and hence also $q H_5$ is contained in $\C[x_1,x_2,x_3,x_5] \setminus 
\C[x_1,x_2,x_3]$, we deduce that $\deg_{x_4} q H_5 = \deg_{x_4} \parder{}{x_5} (q H_5)$.
On account of \eqref{qtnilp} in proposition \ref{qtprop}, we have $\tr \jac qH = 0$, so
\begin{align*}
\deg_{x_4} H_5 &= \deg_{x_4} (qH_5) - \deg_{x_4} q \\
&= \deg_{x_4} {\textstyle\parder{}{x_5}} (q H_5) - \deg_{x_4} q \\
&\le \deg_{x_4} (qH_1,qH_2,qH_3,qH_4) - \deg_{x_4} q \\
&= \deg_{x_4} (H_1,H_2,H_3,H_4) \\
&= \deg h = 3
\end{align*}

Take $k$ minimal such that the leading coefficient with respect to $x_4$ of $p$ is contained in 
$\C[x_1,x_2,\ldots,x_k]$. Since $\deg_{x_4} H_5 \le \deg h$ and $\deg_{x_4,x_5} p = 1$, we have 
$(\parder{}{x_4})^{\deg h + 1} (H_5 \parder{}{x_5} p) = 0$. From \eqref{fxtHeqv} it follows that 
$\jac p \cdot H = 0$, so that we can deduce from $\deg_{x_4} p = \deg_{x_4} q = 1$ that
\begin{align*}
0 - 0 &= \Big(\parder{}{x_4}\Big)^{\deg h + 1} \big(\jac p \cdot H\big) -
\Big(\parder{}{x_4}\Big)^{\deg h + 1} \Big( H_5 \parder{}{x_5} p \Big) \\
&= (\deg h + 1)! \cdot \sum_{i=1}^4 h\Big(\parder{}{x_4} p, \parder{}{x_4} q\Big) 
   \parder{}{x_4} \parder{}{x_i} p \\
&= (\deg h + 1)! \cdot \sum_{i=1}^k h\Big(\parder{}{x_4} p, \parder{}{x_4} q\Big)
   \parder{}{x_4} \parder{}{x_i} p
\end{align*}
But the right hand side has degree $\deg_{y_2} h_k$ with respect to $x_5$ on account of \eqref{x4x5pq}. 
Contradiction, so $\deg q \ge 5$. \qedhere

\end{enumerate}
\end{proof}

\mathversion{bold}
\section{The kernel of the map $H$ of quasi-translations $x + H$} \label{kerqt}
\mathversion{normal}

In the beginning of the proof of theorem \ref{Fallbth}, we have shown that for 
quasi-translations $x + H$ which belong to case b) in \cite{gornoet}, 
$\dim V(H) = \rk \jac H = 3$. Hence the Zariski closure of the image 
of $H$ is an irreducible component of $V(H)$ for such quasi-translations. 
Corollary \ref{hmgrk3cor} in this section subsequently gives us several results
about quasi-translations which belong to case b) in \cite{gornoet}, among which
a result about such quasi-translations without linear invariants.

First we prove some geometric results about quasi-translations to obtain theorem \ref{XW}.
Next, we use theorem \ref{XW} to prove corollary \ref{hmgrk3cor}.
At last, we use theorem \ref{XW} to prove corollary \ref{hmgrk2cor}, which gives us the 
case where $n \ge 6$ of (iii) of theorem \ref{hmgqt5th}.

\begin{lemma} \label{projlem}
Assume $x + H$ is a quasi-translation in dimension $n$ over $\C$. 
Let $X \subseteq \C^n$ be an irreducible variety such that $H|_X$ is not the zero map, 
so that the Zariski closure $Y$ of the image of $H|_X$ is nonzero.

Then for each $c \in X$, there exists a nonzero $p \in Y$ such that $g(c + tp) = g(c)$, 
for every invariant $g$ of $x + H$, where $t$ is a new indeterminate.
\end{lemma}

\begin{proof}
Let $G$ be the set of invariants of $x + H$.
We first prove this lemma for all $c$ in a nonempty open subset of $X$.
The generic property of $c$ that we assume is that $H(c) \ne 0$. Since
$H|_X$ is not the zero map, we are considering a nonempty open subset of $X$
indeed. From \eqref{fxtHeqv} in proposition \ref{qtprop}, it follows that $g(x + tH) = g(x)$
for every invariant $g$ of $x + H$. Hence $g(c + t p) = g(c)$ for every $g \in G$, 
if we take $p = H(c) \ne 0$.

In the general case, consider the sets
$$
Z := \{(c,p,b) \in X \times (\C^n)^2 \mid 
g(c + tp) = g(c) \mbox{ for every $g \in G$ and } b\tp p = 1 \}
$$
and 
$$
\tilde{Z} := \{(c,p,b) \in Z \mid b \mbox{ is the complex conjugate of } p\}
$$
By applying proper substitutions in $t$, we see that the image $\tilde{X}$ of the projection 
of $\tilde{Z}$ onto its first $n$ coordinates is equal to that of $Z$. Since $\tilde{X}$ 
contains an open subset of $X$, it follows from lemma \ref{completeness} that $\tilde{X} = X$, 
which gives the desired result.
\end{proof}

\begin{lemma} \label{Vimlm}
Assume $x + H$ is a quasi-translation in dimension $n$ over $\C$. Let $W$ be the Zariski closure
of the image of $H$. Then for any linear subspace $L$ of $\C^n$, the assertions
\begin{enumerate}[\upshape (1)]

\item $\dim L > \dim V(H)$;

\item every irreducible component of $H^{-1}(L)$ has dimension greater than
$\dim V(H)$;

\item for each $c \in V(H)$, there exists a nonzero $p \in L \cap W$ such that 
$H(c + tp) = 0$;

\end{enumerate}
satisfy {\upshape (1)} $\Rightarrow$ {\upshape (2)} $\Rightarrow$  {\upshape (3)}.
\end{lemma}

\begin{proof} 
Assume that $L$ is a linear subspace of $\C^n$.
\begin{description}
 
\item [(1) \imp (2)]
Notice that $H^{-1}(L)$ is the zero set of $n - \dim L$ linear forms in 
$H_1, H_2, \ldots, \allowbreak H_n$.
By applying \cite[Ch.\@ I, Prop.\@ 7.1]{hartshorne} $n- \dim L - 1$ times, it follows that
every irreducible component of $H^{-1}(L)$ has dimension at least $\dim L$, which exceeds
$\dim V(H)$ if (1) is satisfied.

\item [(2) \imp (3)]
Assume $H(c) = 0$. Since $V(H) = H^{-1}(0) \subseteq H^{-1}(L)$, 
there exists an irreducible component $X$ of $H^{-1}(L)$ which contains $c$. Assuming (2),
we obtain $\dim X > \dim V(H)$, whence $H|_X$ is not the zero map. 
Hence (2) $\Rightarrow$ (3) follows from lemma \ref{projlem} and (1) $\Rightarrow$ (2)
of proposition \ref{qtprop}. \qedhere

\end{description}
\end{proof}

\begin{theorem} \label{XW}
Assume $x + H$ is a homogeneous quasi-translation over $\C$. Let 
$W$ be the Zariski closure of the image of $H$.
 
Then for every irreducible component $X$ of $V(H)$ such that
$\dim (X \cap W) \le n - \dim V(H)$, $X \cap W$ has an irreducible
component $Z$ of dimension $n - \dim V(H)$, such that $c + p \in X$
for all $c \in X$ and all $p$ in the linear span of $Z$.
\end{theorem}

\begin{proof}
Let $X$ be an irreducible component of $V(H)$ such that
$\dim (X \cap W) \le n - \dim V(H)$. We can take a linear subspace
$M$ of $\C^n$, such that $c + p \subseteq X$ 
for all $c \in X$ and all $p \in M$, because $M = \{0\}^n$ suffices. 
Take $M$ as above such that $\dim M$ is as large as possible.
Suppose first that $\dim (M \cap X \cap W) = n - \dim V(H)$. 
Then $\dim (X \cap W) = n - \dim V(H)$ as well, so that $X \cap W$ 
has an irreducible component $Z \subseteq M$ of maximum dimension $n - \dim V(H)$. 
Since $M$ contains the linear span of $Z$, it follows from the definition of $M$ 
that $Z$ suffices.

Suppose next that $\dim (M \cap X \cap W) < n - \dim V(H)$. Take for $L$ a generic 
linear subspace of $\C^n$ of dimension $\dim V(H) + 1$, to obtain that 
$\dim \big(L \cap (M \cap X \cap W)\big) = 0$ and $\dim \big(L \cap (X \cap W)\big) \le 1$.
Since $X$ is an irreducible component of $V(H)$, the interior $X^{\circ}$ of $X$ as 
a closed subset of $V(H)$ is nonempty. 
Now take an arbitrary $c \in X^{\circ}$. On account of (1) $\Rightarrow$ (3) of lemma 
\ref{Vimlm}, there exists a nonzero $p \in L \cap W$, such that $H(c + t p) = 0$. Since $H$ is 
homogeneous, the set $L \cap W$ is a union of lines through the origin. Hence there exists a 
line $P \subseteq L \cap W$ though the origin, such that $c + P \subseteq V(H)$. 

Since $c \in X^{\circ}$ and $X$ is an irreducible component of $V(H)$, we deduce that 
$c + P \subseteq X$. In particular, $P \subseteq X$, so $P \subseteq L \cap X \cap W$.
But $\dim (L \cap X \cap W) \le 1$, so $L \cap X \cap W$ can only contain finitely
many lines through the origin, say that $Q$ is the finite set of these lines. Since
$X^{\circ}$ is dense in $X$ and $c$ was arbitrary, we can deduce that
$$
X = \bigcup_{P \in Q} \{ c \in X \mid c + P \subseteq X \}
$$
Since $X$ is irreducible, there exists a $P \in Q$ such that $c + P \subseteq X$ for all $c \in X$. 
Therefore we can replace $M$ by $M \oplus P$, which contradicts the maximality of $\dim M$.
\end{proof}

\begin{corollary} \label{hmgrk3cor}
Assume $x + H$ is a homogeneous quasi-translation over $\C$, such that 
$\dim V(H) \le 3$ and $\gcd\{H_1,H_2,\ldots,H_n\} = 1$.
 
Then the Zariski closure $W$ of the image of $H$ is contained in $V(H)$.
Furthermore, every irreducible component $X$ of $V(H)$ which is not equal to $W$ 
is a $3$-dimensional linear subspace of $\C^n$ for which $\dim (X \cap W) = 2$.

If $W$ has a nonzero (projective) apex $p$ and $V(H)$ has a component $X$ which does not contain
$p$, then $W$ is contained in the $4$-dimensional linear subspace of $\C^n$ which is 
spanned by $X$ and $p$.
\end{corollary}

\begin{proof}
Using (2) $\Rightarrow$ (3) of proposition \ref{hmgprop} and lemma \ref{Wrk}, we deduce that 
$W$ is irreducible and that $W \subseteq V(H)$. 
Let $X$ be an irreducible component of $V(H)$ which is not equal to $W$.
Since $X \ne W$ and $\dim V(H) \le 3$, we have $\dim (X \cap W) \le 2$.
From $\gcd\{H_1,H_2,\ldots, H_n\} = 1$, we deduce that
$\dim (X \cap W) \le 2 \le n - \dim V(H)$. On account of theorem \ref{XW}, $X \cap W$ has an 
irreducible component $Z$ of dimension $n - \dim V(H) = 2 = \dim (X \cap W)$, such that 
$c + q \in X$ for all $c \in X$ and all $q$ in the linear span of $Z$.

Notice that $\dim X \le \dim V(H) \le 3$.
Suppose that $\dim X \le 2$. Then $X \subseteq W$ because $X$ is irreducible
and $\dim (X \cap W) = 2$. Since $W$ is irreducible, this contradicts the fact that  
$X$ is an irreducible component of $V(H)$ which is not equal to $W$. Thus $\dim X = 3$.
Let $r$ be the dimension of the linear span of $Z$. If $r \ge 3$, then 
$X$ contains the linear span of $r$ independent $q \in Z$, whence 
$X$ is equal to the linear span of $r = 3$ 
independent $q \in Z$. If $r \le 2$, then $r = 2$ because 
$\dim Z = 2$, and $X$ is the linear span of two independent $q \in Z$,
and any $c \in X \setminus Z$.

Suppose that $W$ has a nonzero (projective) apex $p$ and $V(H)$ has a component $X$
which does not contain $p$. Since $\dim (X \cap W) = 2$, there are infinitely many
GN-planes spanned by $p$ and a nonzero $q \in X \cap W$. Any proper algebraic subset
of $W$ can only have finitely many GN-planes, because $W$ is irreducible and $\dim W = 3$.
Hence the set of infinitely many GN-planes spanned by $p$ and a nonzero $q \in X \cap W$ is dense
in $W$. It follows that $W$ is contained in the linear span of $X$ and $p$.
\end{proof}

\begin{corollary} \label{hmgrk2cor}
Assume $x + H$ is a homogeneous quasi-translation over $\C$, such that 
$\rk \jac H + \dim V(H) \le n$. Then $H(c + p) = 0$
for all $c \in V(H)$ and all $p$ in the linear span of the image of $H$.
In particular, $x + H$ has at least $\rk \jac H$ linear invariants.
\end{corollary}

\begin{proof}
The case where $\deg H \le 0$ is easy, so assume that $\deg H \ge 1$.
Let $W$ be the Zariski closure of the image of $H$ and $X$ be an irreducible component of $V(H)$.
From lemma \ref{Wrk}, it follows that $W$ is irreducible and that
$\dim (X \cap W) \le \dim W = \rk \jac H \le n - \dim V(H)$.
Using theorem \ref{XW}, we subsequently deduce that $X \cap W$ has an irreducible component $Z$ of 
dimension $n - \dim V(H)$, such that $c + p \in X$ for all $c \in X$ and all $p$ in the linear span of $Z$. 

If $W \nsubseteq X$, then by the irreducibility of $W$, $\dim Z \le \dim (X \cap W) < \dim W = 
\rk \jac H \le n - \dim V(H)$, which contradicts $\dim Z = n - \dim V(H)$. Hence $W \subseteq X$, 
and by irreducibility of $W$ once again, the only irreducible component of $X \cap W$ is $W$. 
Thus $Z = W$. Furthermore, $X$ is an arbitrary irreducible component of $V(H)$, so $c + p \in V(H)$ 
for all $c \in V(H)$ and all $p$ in the linear span of $W$. 

Consequently, $H(c + p) = 0$ for all $c \in V(H)$ and all $p$ 
in the linear span of the image of $H$. Furthermore, the dimension of the linear span of the image 
of $H$ does not exceed the dimension of $V(H)$. So there are at least $r := n - \dim V(H) \ge 
\rk \jac H$ independent linear forms $l_1, l_2, \ldots, l_r$ which vanish on the image of $H$.
Hence $l_i(H) = 0$ and $l_i(x + H) = l_i(x)$ for each $i$, as desired.
\end{proof}


\begin{thebibliography}{19}

\bibitem{singhess}
M.C. de Bondt and A.R.P. van de Essen, Singular Hessians, J. Algebra 282 (2004), 
no.\@ 1, 195--204. 

\bibitem{dp3}
M.C. de Bondt and A.R.P. van den Essen,
The Jacobian conjecture: linear triangularization for homogeneous polynomial maps in 
dimension three. J. Algebra 294 (2005), no.\@ 1, 294--306. 

\bibitem{debunk}
M.C. de Bondt,
Quasi-translations and counterexamples to the homogeneous dependence problem. 
Proc.\@ Amer.\@ Math.\@ Soc.\@ 134 (2006), no.\@ 10, 2849--2856 (electronic).

\bibitem{homokema}
M.C. de Bondt,
Homogeneous Keller maps, Ph.D. thesis, Radboud University Nijmegen, July 2009. \\
\verb+http://webdoc.ubn.ru.nl/mono/b/bondt_m_de/homokema.pdf+

\bibitem{strongnil}
Michiel de Bondt, 
The strong nilpotency index of a matrix. 
Linear Multilinear Algebra 62 (2014), no.\@ 4, 486--497. 

\bibitem{cilrussim}
C. Ciliberto, F. Russo and A. Simis, 
Homaloidal hypersurfaces and hypersurfaces with vanishing Hessian. 
Adv.\@ Math.\@ 218 (2008), 1759--1805.

\bibitem{arnobook}
A. van den Essen, Polynomial Automorphisms and the Jacobian Conjecture.
Progress in Mathematics, no.\@ 190, Birkh{\"a}user, Berlin, 2000.

\bibitem{esshubnc}
A.R.P. van den Essen and E.-M.G.M. Hubbers, 
A new class of invertible polynomial maps.
J. Algebra 187 (1997), no.\@ 1, 214--226.

\bibitem{franch}
Alfredo Franchetta, Sulle forme algebriche di $S_4$ aventi l'hessiana indeterminata. 
Rend.\@ Mat.\@ e Appl.\@ (5) 14 (1954), 252--257. 

\bibitem{garrep}
Alice Garbagnati and Flavia Repetto, A geometrical approach to Gordan-N{\"o}ther's 
and Franchetta's contributions to a question posed by Hesse. 
Collect.\@ Math.\@ 60 (2009), no.\@ 1, 27--41. 

\bibitem{gornoet}
P. Gordan and M. N{\"o}ther, {\"U}ber die algebraische Formen,
deren Hes\-se'sche Determinante identisch verschwindet.
Math.\@ Ann.\@ 10 (1876), 547--568.

\bibitem{hartshorne}
Robin Hartshorne, Algebraic geometry. 
Graduate Texts in Mathematics, no.\@ 52, Springer-Verlag, New York-Heidelberg, 1977. 

\bibitem{liu}
Dayan Liu, On homogeneous quasi-translations in dimension five. 
Int. Math. Forum 6 (2011), no. 53--56, 2655--2664. 

\bibitem{gnlossen}
C. Lossen, When does the Hessian determinant vanish identically? 
(On Gordan and N{\"o}ther's proof of Hesse's claim). 
Bull.\@ Braz.\@ Math.\@ Soc.\@ (N.S.) 35 (2004), no.\@ 1, 71--82. 

\bibitem{advanced}
Kevin C.\@ O'Meara, John Clark and Charles I. Vinsonhaler, 
Advanced topics in linear algebra. Weaving matrix problems through the Weyr form. 
Oxford University Press, Oxford, 2011. 
 
\bibitem{redbook}
David Mumford, The red book of varieties and schemes. 
Second, expanded edition. Includes the Michigan lectures (1974) on curves and their Jacobians. 
With contributions by Enrico Arbarello. 
Lecture Notes in Mathematics, no.\@ 1358, Springer-Verlag, Berlin, 1999. 

\bibitem{russo}
Francesco Russo, On the Geometry of Some Special Projective Varieties.
Lecture Notes of the Unione Matematica Italiana, Springer, 2016.

\bibitem{watanabe}
Junzo Watanabe, On the theory of Gordan--Noether on homogeneous forms with zero Hessian. 
Proc.\@ Sch.\@ Sci.\@ Tokai Univ.\@ 49 (2014), 1--21.

\end{thebibliography}
\end{document}